\documentclass[11pt]{amsart}

\usepackage{graphics}
\usepackage{color}
\usepackage[a4paper, margin=3cm]{geometry}

\usepackage{amssymb,enumerate}
\usepackage{amsmath}
\usepackage{bbm}           
\usepackage{bm}
\usepackage{eso-pic,graphicx}
\usepackage{tikz}
\usepackage{cite}
\usepackage{esint}
\usepackage[colorlinks=true, pdfstartview=FitV, linkcolor=blue, citecolor=blue, urlcolor=blue]{hyperref}

\usepackage{graphicx}

\usepackage{amsmath,amsfonts}
\usepackage{rotating}
\usepackage{amssymb}
\usepackage{verbatim}
\usepackage{rotating}
\usepackage{mathrsfs}
\begin{document}

\newtheorem{theorem}{Theorem}   

\newtheorem{proposition}[theorem]{Proposition}

\newtheorem{conjecture}[theorem]{Conjecture}
\def\theconjecture{\unskip}
\newtheorem{corollary}[theorem]{Corollary}
\newtheorem{lemma}[theorem]{Lemma}
\newtheorem{sublemma}[theorem]{Sublemma}
\newtheorem{observation}[theorem]{Observation}
\newtheorem{remark}[theorem]{Remark}
\newtheorem{definition}[theorem]{Definition}

\theoremstyle{definition}

\newtheorem{notation}[theorem]{Notation}

\newtheorem{question}[theorem]{Question}
\newtheorem{questions}[theorem]{Questions}
\newtheorem{example}[theorem]{Example}
\newtheorem{problem}[theorem]{Problem}
\newtheorem{exercise}[theorem]{Exercise}

\numberwithin{theorem}{section} \numberwithin{theorem}{section}
\numberwithin{equation}{section}

\def\earrow{{\mathbf e}}
\def\rarrow{{\mathbf r}}
\def\uarrow{{\mathbf u}}
\def\varrow{{\mathbf V}}
\def\tpar{T_{\rm par}}
\def\apar{A_{\rm par}}

\def\reals{{\mathbb R}}
\def\torus{{\mathbb T}}
\def\heis{{\mathbb H}}
\def\integers{{\mathbb Z}}
\def\naturals{{\mathbb N}}
\def\complex{{\mathbb C}\/}
\def\distance{\operatorname{distance}\,}
\def\support{\operatorname{support}\,}
\def\dist{\operatorname{dist}\,}
\def\Span{\operatorname{span}\,}
\def\degree{\operatorname{degree}\,}
\def\kernel{\operatorname{kernel}\,}
\def\dim{\operatorname{dim}\,}
\def\codim{\operatorname{codim}}
\def\trace{\operatorname{trace\,}}
\def\Span{\operatorname{span}\,}
\def\dimension{\operatorname{dimension}\,}
\def\codimension{\operatorname{codimension}\,}
\def\nullspace{\scriptk}
\def\kernel{\operatorname{Ker}}
\def\ZZ{ {\mathbb Z} }
\def\p{\partial}
\def\rp{{ ^{-1} }}
\def\Re{\operatorname{Re\,} }
\def\Im{\operatorname{Im\,} }
\def\ov{\overline}
\def\eps{\varepsilon}
\def\lt{L^2}
\def\diver{\operatorname{div}}
\def\curl{\operatorname{curl}}
\def\etta{\eta}
\newcommand{\norm}[1]{ \|  #1 \|}
\def\expect{\mathbb E}
\def\bull{$\bullet$\ }
\def\C{\mathbb{C}}
\def\R{\mathbb{R}}
\def\Rn{{\mathbb{R}^n}}
\def\Sn{{{S}^{n-1}}}
\def\M{\mathbb{M}}
\def\N{\mathbb{N}}
\def\Q{{\mathbb{Q}}}
\def\Z{\mathbb{Z}}
\def\F{\mathcal{F}}
\def\L{\mathcal{L}}
\def\S{\mathcal{S}}
\def\supp{\operatorname{supp}}
\def\dist{\operatorname{dist}}
\def\essi{\operatornamewithlimits{ess\,inf}}
\def\esss{\operatornamewithlimits{ess\,sup}}
\def\xone{x_1}
\def\xtwo{x_2}
\def\xq{x_2+x_1^2}
\newcommand{\abr}[1]{ \langle  #1 \rangle}

\newcommand{\Norm}[1]{ \left\|  #1 \right\| }
\newcommand{\set}[1]{ \left\{ #1 \right\} }

\def\one{\mathbf 1}
\def\whole{\mathbf V}
\newcommand{\modulo}[2]{[#1]_{#2}}

\def\scriptf{{\mathcal F}}
\def\scriptg{{\mathcal G}}
\def\scriptm{{\mathcal M}}
\def\scriptb{{\mathcal B}}
\def\scriptc{{\mathcal C}}
\def\scriptt{{\mathcal T}}
\def\scripti{{\mathcal I}}
\def\scripte{{\mathcal E}}
\def\scriptv{{\mathcal V}}
\def\scriptw{{\mathcal W}}
\def\scriptu{{\mathcal U}}
\def\scriptS{{\mathcal S}}
\def\scripta{{\mathcal A}}
\def\scriptr{{\mathcal R}}
\def\scripto{{\mathcal O}}
\def\scripth{{\mathcal H}}
\def\scriptd{{\mathcal D}}
\def\scriptl{{\mathcal L}}
\def\scriptn{{\mathcal N}}
\def\scriptp{{\mathcal P}}
\def\scriptk{{\mathcal K}}
\def\frakv{{\mathfrak V}}

\newtheorem*{remark0}{\indent\sc Remark}

\renewcommand{\proofname}{\indent\sc Proof.}

\title[Sparse domination and weighted variation inequalities]
{Sparse dominations and weighted variation inequalities for singular integrals and commutators}

\author{Yongming Wen, Huoxiong Wu$^*$ and Qingying Xue}


\thanks{$^*$Corresponding author.}
\thanks{Supported by the NNSF of China (Nos. 11771358, 11871101) and NSFC-DFG (No. 11761131002).}

\address{School of Mathematical Sciences, Xiamen University, Xiamen 361005, China} \email{wenyongmingxmu@163.com}
\address{School of Mathematical Sciences, Xiamen University, Xiamen 361005, China} \email{huoxwu@xmu.edu.cn}
\address{School of Mathematical Sciences, Beijing Normal University; Laboratory of Mathematics and Complex Systems,
	Ministry of Education, Beijing 100875, China} \email{qyxue@bnu.edu.cn}

\begin{abstract}
	
	This paper gives the pointwise sparse dominations for variation operators of singular integrals and commutators with kernels satisfying the $L^r$-H\"{o}rmander conditions. As applications, we obtain the strong type quantitative weighted bounds for such variation operators as well as the weak-type quantitative weighted bounds for the variation operators of singular integrals and the quantitative weighted weak-type endpoint estimates for variation operators of commutators, which are completely new even in the unweighted case. In addition, we also obtain the local exponential decay estimates for such variation operators.
	
\end{abstract}
\maketitle

\section{Introduction and main results}

During the past few years, a novel set of approaches that allow to dominate operators by sparse operators has blossomed. It provides us with a new way to simplify the proofs of known results or to draw new conclusions in the weights theory. The sparse operators were originally  introduced and used by Lerner \cite{Le} to simplify the proof of the $A_2$ Conjecture \cite{Hy1}. Later on, the following sparse domination for an $\omega$-Calder\'{o}n-Zygmund operator $T$ was given in \cite{CR} and \cite{LN}, independently,
\begin{equation}\label{(1.1)}
	|Tf(x)|\leq c_n\kappa_{T}\sum_{j=1}^{3^n}\mathcal{A}_{S_j}f(x),
\end{equation}
where $\mathcal{A}_{S}f(x)=\sum_{Q\in S}|Q|^{-1}\int_{Q}|f|\chi_Q(x)$ and $S$ is a sparse family of dyadic cubes from $\mathbb{R}^n$ (see \cite{LN} for the definition of $T$). Since then, there is a good number of literature using sparse domination methods to deal with other operators, see \cite{BFP,CLPR,CCPO,HL,LS}, and these are far from complete.

Now let's turn to the commutators of a linear or sublinear operator $T$.
Recall first that a locally integrable function $b$ is said in  $\rm BMO(\mathbb{R}^n)$ spaces if
$$\|b\|_{\rm BMO}:=\sup_{Q}\frac{1}{|Q|}\int_{Q}|b-b_Q|<\infty,$$
where $b_Q=|Q|^{-1}\int_{Q}b(x)dx$.
Then the commutator $[b,T]$ generated by $T$ with $b$ is defined by
$$[b,T]f(x):=b(x)T(f)(x)-T(bf)(x).$$
In 2017, Lerner, Ombrosi and Rivera-R\'{i}os \cite{LOR1} obtained an analogue of $(\ref{(1.1)})$ for commutators of $\omega$-Calder\'{o}n-Zygmund operators, in which they gave several weighted weak type bounds for $[b,T]$ and the quantitative two-weighted estimate for $[b,T]$ due to Holmes et al. \cite{HLW}. Subsequently, Lerner et al. \cite{LOR2} extended the results  to the iterated commutator and established the necessary conditions for a rather wider class of operators. For more applications of sparse operators to commutators, see \cite{AMR,CLPR,IR,Ort,PP,PR,Ri}, and references therein.

In this paper, we will focus on the variation operators of singular integrals and their commutators. These operators can be used to measure the pointwise rate of convergence for the truncated versions of singular integrals and the corresponding commutators. Moreover, they are pointwisely larger with stronger degree of nonlinearity than maximal truncated singular integrals and their commutators, respectively. It has been shown that the variation operators for martingales and several families of operators are closely connected with probability, ergodic theory and harmonic analysis. For more earlier results, we refer the readers to \cite{AJS,Bou,Jo,JR}.

Before stating our results, we first recall some definitions and backgrounds.

\begin{definition} Let $\mathcal{T}=\{T_\varepsilon\}_{\varepsilon>0}$ be a family of operators and the  $\lim_{\varepsilon\rightarrow 0}T_{\varepsilon}f(x)=Tf(x)$ exists in some  sense. The $\rho$-variation operator is defined as
	$$\mathcal{V}_\rho(\mathcal{T}f)(x):=\sup_{\varepsilon_i\downarrow0}\big(\sum_{i=1}^{\infty}|
	T_{\varepsilon_{i+1}}f(x)-T_{\varepsilon_i}f(x)|^\rho\Big)^{1/\rho},$$
	where $\rho>1$ and the supremum is taken over all sequences $\{\varepsilon_i\}$ decreasing to zero.
\end{definition}

\begin{definition}\label{def1.2}
	The truncated singular integral operators $\mathcal{T}:=\{T_\varepsilon\}_{\varepsilon>0}$ and commutators $\mathcal{T}_b:=\{T_{\varepsilon,b}\}_{\varepsilon>0}$ are given by
	\begin{equation}\label{(1.2)}
		T_\varepsilon(f)(x):=\int_{|x-y|>\varepsilon}K(x,y)f(y)dy,
	\end{equation}
	and
	\begin{equation}\label{(1.3)}
		T_{\varepsilon,b}(f)(x):=\int_{|x-y|>\varepsilon}[b(x)-b(y)]K(x,y)f(y)dy,
	\end{equation}
	where $b\in \rm BMO(\mathbb{R}^n)$ and $K(x,y)$ satisfies the size condition:
	\begin{equation}\label{(1.4)}
		|K(x,y)|\leq C/|x-y|^n
	\end{equation}
	and $L^r$-H\"{o}rmander condition (denote by $K\in \mathcal{H}_r$):
	$$\sup_{Q}\sup_{x,z\in\frac{1}{2}Q}\sum_{k=1}^{\infty}|2^kQ|
	\Big(\frac{1}{|2^kQ|}\int_{2^kQ\backslash 2^{k-1}Q}|K(x,y)-K(z,y)|^r dy\Big)^{1/r}<\infty,$$
	and
	$$\sup_{Q}\sup_{x,z\in\frac{1}{2}Q}\sum_{k=1}^{\infty}|2^kQ|
	\Big(\frac{1}{|2^kQ|}\int_{2^kQ\backslash 2^{k-1}Q}|K(y,x)-K(y,z)|^r dy\Big)^{1/r}<\infty,$$
	when $1\leq r<\infty$. For $r=\infty$, we mean that
	$$\sup_{Q}\sup_{x,z\in\frac{1}{2}Q}\sum_{k=1}^{\infty}|2^kQ|
	\esss_{y\in2^kQ\backslash 2^{k-1}Q}|K(x,y)-K(z,y)|<\infty,$$
	and
	$$\sup_{Q}\sup_{x,z\in\frac{1}{2}Q}\sum_{k=1}^{\infty}|2^kQ|
	\esss_{y\in2^kQ\backslash 2^{k-1}Q}|K(y,x)-K(y,z)|<\infty.$$
\end{definition}

If we denote the class of kernels of $\omega$-Calder\'{o}n-Zygmund operators by $\mathcal{H}_{Dini}$, then for $1<s<r<\infty$, it is easy to check that (also see \cite{IR})
$$\mathcal{H}_{Dini}\subset\mathcal{H}_\infty\subset\mathcal{H}_r\subset\mathcal{H}_s
\subset\mathcal{H}_1.$$

In 2000, Campbell et al. \cite{CJRW1} first proved that the $\rho$-variation operators for Hilbert transform are strong $(p,p)$ and weak $(1,1)$ types if $\rho>2$. Subsequently, the aforementioned authors \cite{CJRW2} extended the results to the higher dimensional cases, such as Riesz transforms and the homogeneous singular integrals with rough kernels. For further results of the cases of rough kernels, we refer the readers to \cite{DHL}. For the weighted cases, we refer to \cite{CMTT,GT,HLP,MaTX1,MaTX2}. In particularly, Hyt\"{o}nen, Lacey and P\'{e}rez \cite{HLP} gave the sharp weighted bounds for the $\rho$-variation of Calder\'{o}n-Zygmund operators satisfying a $\log$-Dini condition. See \cite{BMSW,DuLY,MasT,MiST} etc. for other recent works on variational inequalities and their applications.

On the other hand, the variation inequalities for the commutators of singular integrals have also attracted several authors attentions. In 2013, Betancor et al. \cite{BFHR} obtained the $L^p$-boundedness of variation operators for the commutators of Riesz transforms in Euclidean setting and Schr\"{o}dinger setting. In 2015, Liu and Wu \cite{LW} studied the boundedness for $\mathcal{V}_{\rho}(\mathcal{T}_b)$ on the weighted $L^p$ spaces with $p>1$ and $\rho>2$, where $b\in {\rm BMO}$ and the kernel of the $T_\varepsilon$ is a standard Calder\'{o}n-Zygmund kernel. Following this work, Zhang and Wu \cite{ZW} established the weighted strong type bounds for the variation operators of commutators with kernels satisfying certain H\"{o}rmander conditions. Recently, the variation inequalities for commutators of singular integrals with rough kernels were also established in \cite{CDHL}. However, to our knowledge, there is no any results on the weak-type endpoint estimates for the variation of commutators, moreover, on the quantitative weighted bounds.

Inspired by the above works, this paper aims to extend the quantitative weighted results for variation of $\mathcal{T}$ with kernels from $\mathcal{H}_{\rm Dini}$ into $\mathcal{H}_r$ for $1<r\le \infty$, and to establish the quantitative weighted variation inequalities for the families of commutators $\mathcal{T}_b$ and the corresponding weighted endpoint estimate, which is completely new even in the un-weighted case.

Our main ingredients are to establish the sparse dominations for variation operator of singular integrals and commutators, which are non-trivial, especially in proving the weak type estimate of a local grand maximal truncated variation operator, and due to the methods used in \cite{LOR2} do not work and to avoid employing the trick of Cauchy integral formula as in \cite{IR}, we seek for appropriate ways relied on sparse dominations to obtain the quantitative weighted bounds and weak-type endpoint estimates for variations of commutators. Now our main results can be formulated as follows.

\begin{theorem}\label{Theorem1.1}\,\
	Let $1<r\leq\infty$, $\rho>2$ and $b\in L_{loc}^{1}(\mathbb{R}^n)$. Let $\mathcal{T}=\{T_\varepsilon\}_{\varepsilon>0}$ and $\mathcal{T}_{b}=\{T_{\varepsilon,b}\}_{\varepsilon>0}$ be given by $(\ref{(1.2)})$ and $(\ref{(1.3)})$, respectively. Assume that the kernel $K(x,y)\in \mathcal{H}_r$ and satisfies $(\ref{(1.4)})$. If $\mathcal{V}_\rho(\mathcal{T})$ is bounded on $L^{q_0}(\mathbb{R}^n,dx)$ for some $1<q_0<\infty$, then for every $f\in C_{c}^{\infty}(\mathbb{R}^n)$, there exist $3^n$ sparse families $\mathcal{S}_j$ such that
	\begin{equation}\label{(1.5)}
		\mathcal{V}_{\rho}(\mathcal{T}f)(x)\leq C(n,r',q_0,\|\mathcal{V}_\rho(\mathcal{T})\|_{L^{q_0}\rightarrow L^{q_0}})
		\sum_{j=1}^{3^n}\sum_{R\in \mathcal{S}_j}\langle|f|^{r'}\rangle_{R}^{1/r'}\chi_{R}(x),
	\end{equation}
	and
	\begin{equation}\label{(1.6)}
		\begin{split}
			\mathcal{V}_{\rho}(\mathcal{T}_bf)(x)&\leq C(n,r',q_0,\|\mathcal{V}_\rho(\mathcal{T})\|_{L^{q_0}\rightarrow L^{q_0}})\\ &\quad\times\sum_{j=1}^{3^n}\sum_{R\in \mathcal{S}_j}\Big\{|b(x)-\langle b\rangle_R|\langle|f|^{r'}\rangle_{R}^{1/r'}+
			\langle|f(b-\langle b\rangle_R)|^{r'}\rangle_{R}^{1/r'}\Big\}\chi_R(x),
		\end{split}
	\end{equation}
	where $\langle |f|^{r'}\rangle_{R}=\frac{1}{|R|}\int_{R}|f(y)|^{r'}dy$.
\end{theorem}

\begin{remark}In \cite{FZ}, Franca Silva and Zorin-Kranich applied the sparse domination to explore the sharp weighted estimates for the variation operators associated with the family of $\omega$-Calderon-Zygmund operators. Our results can be regarded as a generalization of the results in \cite{FZ}, since $\mathcal{H}_{Dini}\subset\mathcal{H}_r$.
\end{remark}

Applying the conclusion $(\ref{(1.5)})$ in Theorem $\ref{Theorem1.1}$, we get the following sharp weighted estimates for the variation operators of singular integrals.

\begin{theorem}\label{Theorem1.2}\,\
	Let $1<r\leq\infty$, $\rho>2$ and $\omega$ and $\sigma^{r'}$ be a pair of weights. Let $\mathcal{T}=\{T_\varepsilon\}_{\varepsilon>0}$ be given by $(\ref{(1.2)})$. Assume that $K(x,y)\in \mathcal{H}_r$ and satisfies $(\ref{(1.4)})$. If $\mathcal{V}_\rho(\mathcal{T})$ is bounded on $L^{q_0}(\mathbb{R}^n,dx)$ for some $1<q_0<\infty$, then for any $r'<p<\infty$,
	\begin{align*}
		\|\mathcal{V}_{\rho}(\mathcal{T}\sigma f)\|_{L^p(\omega)}
		&\leq C(n,r',p,q_0,\|\mathcal{V}_\rho(\mathcal{T})\|_{L^{q_0}\rightarrow L^{q_0}})[\omega,\sigma^{r'}]_{A_{p/r'}}^{1/p}\\
		&\qquad\times\Big([\omega]_{A_\infty}^{r'/p'}+[\sigma^{r'}]_{A_\infty}^{r'/p}\Big)
		^{1/r'}\|f\|_{L^p(\sigma^{r'})},
	\end{align*}
	$$
	\|\mathcal{V}_{\rho}(\mathcal{T}\sigma f)\|_{L^{p,\infty}(\omega)}
	\leq C(n,r',p,q_0,\|\mathcal{V}_\rho(\mathcal{T})\|_{L^{q_0}\rightarrow L^{q_0}}) [\omega,\sigma^{r'}]_{A_{p/r'}}^{1/p}[\omega]_{A_\infty}^{1/p'}\|f\|_{L^p(\sigma^{r'})}.
	$$
\end{theorem}

\begin{corollary}\label{Corollary1.3}\,\
	Let $1<r\leq\infty$, $\rho>2$, $\mathcal{T}=\{T_\varepsilon\}_{\varepsilon>0}$ be given by $(\ref{(1.2)})$. Assume that $K(x,y)\in \mathcal{H}_r$ and satisfies $(\ref{(1.4)})$. If $\mathcal{V}_\rho(\mathcal{T})$ is bounded on $L^{q_0}(\mathbb{R}^n,dx)$ for some $1<q_0<\infty$, then for any $r'<p<\infty$, $\omega\in A_{p/r'}$,
	\begin{align*}
		\|\mathcal{V}_{\rho}(\mathcal{T}f)\|_{L^p(\omega)}
		&\leq C(n,r',p,q_0,\|\mathcal{V}_\rho(\mathcal{T})\|_{L^{q_0}\rightarrow L^{q_0}}) [\omega]_{A_{p/r'}}^{1/p}\\
		&\qquad\times\Big([\omega]_{A_\infty}^{r'/p'}+[\omega^{\frac{1}{1-p/r'}}]_{A_\infty}^{r'/p}\Big)^{1/r'}\|f\|_{L^p(\omega)},
	\end{align*}
	$$\|\mathcal{V}_{\rho}(\mathcal{T}f)\|_{L^{p,\infty}(\omega)}\leq C(n,r',p,q_0,\|\mathcal{V}_\rho(\mathcal{T})\|_{L^{q_0}\rightarrow L^{q_0}}) [\omega]_{A_{p/r'}}^{1/p}[\omega]_{A_\infty}^{1/p'}
	\|f\|_{L^p(\omega)}.$$
\end{corollary}

Using $(\ref{(1.5)})$ again and the same method as in \cite{IR}, we can obtain the following weak type estimate for variation operators.
\begin{corollary}\label{Corollary1.4}\,\
	Let $1<r\leq\infty$, $\rho>2$, $\mathcal{T}=\{T_\varepsilon\}_{\varepsilon>0}$ be given by $(\ref{(1.2)})$. Assume that $K(x,y)\in \mathcal{H}_r$ and satisfies $(\ref{(1.4)})$. If $\mathcal{V}_\rho(\mathcal{T})$ is bounded on $L^{q_0}(\mathbb{R}^n,dx)$ for some $1<q_0<\infty$, then for every weight $\omega$ and every Young function $\varphi$,
	\begin{align*}
		\omega(\{x\in \mathbb{R}^n:\mathcal{V}_\rho(\mathcal{T}f)(x)>\lambda\})
		&\leq C(n,r',q_0,\|\mathcal{V}_\rho(\mathcal{T})\|_{L^{q_0}\rightarrow L^{q_0}})\kappa_\varphi\\
		&\qquad\times\int_{R^n}\Big(\frac{|f(x)|}{\lambda}\Big)^{r'}M_\varphi\omega(x)dx,
	\end{align*}
	where
	$$\kappa_\varphi:=\int_{1}^{\infty}\frac{\varphi^{-1}(t)[\log(e+t)]^{2r'}}{t^2[\log(e+t)]^3}dt.$$
\end{corollary}

\begin{remark}
	We remark that the first conclusions in Theorem $\ref{Theorem1.2}$ and Corollary $\ref{Corollary1.3}$ improve the main results in \cite{HLP} by removing the weak $(1,1)$ type assumption and weakening the condition of kernel, and the second conclusions are new. Therefore, Theorem $\ref{Theorem1.2}$ and Corollary $\ref{Corollary1.3}$ can be regarded as an extension and generalization of the main results in \cite{HLP}. In Corollary $\ref{Corollary1.4}$, take $r'=1$, $\varphi(t)=t$ and $\omega\in A_1$, since $\mathcal{H}_{Dini}\subset\mathcal{H}_\infty$, our argument covers the result in \cite{DuLY}, moreover, the conclusion in Corollary $\ref{Corollary1.4}$ is new itself.
\end{remark}

Moreover, applying the conclusion $(\ref{(1.6)})$ in Theorem $\ref{Theorem1.1}$, we can obtain the following quantitative weighted bounds for the variation operators of commutators of singular integrals, which are also completely new.

\begin{theorem}\label{Theorem1.4}\,\
	Let $1<r\leq\infty$, $\rho>2$, $b\in {\rm BMO}(\mathbb{R}^n)$. Assume that $K(x,y)\in \mathcal{H}_r$ and satisfies $(\ref{(1.4)})$. Let $\mathcal{T}$ and $\mathcal{T}_{b}$ be given by $(\ref{(1.2)})$ and $(\ref{(1.3)})$, respectively. If $\mathcal{V}_\rho(\mathcal{T})$ is bounded on $L^{q_0}(\mathbb{R}^n,dx)$ for some $1<q_0<\infty$, then for any $r'<p<\infty$, $\omega\in A_{p/r'}$,
	\begin{equation*}
		\begin{split}
			\|\mathcal{V}_{\rho}(\mathcal{T}_bf)\|_{L^p(\omega)}
			&\leq C(n,r',p,q_0,\|\mathcal{V}_\rho(\mathcal{T})\|_{L^{q_0}\rightarrow L^{q_0}}) [\omega]_{A_\infty}^{2}\\
			&\quad\times\Big([\omega]_{A_{p/r'}}^{\frac{p+r'}{p(p-r')}}+([\omega]_{A_{p/r'}}
			[\omega^{-\frac{r'}{p-r'}}]_{A_\infty})^{1/p}\Big)\|b\|_{\rm BMO}\|f\|_{L^p(\omega)}.
		\end{split}
	\end{equation*}
\end{theorem}

\begin{theorem}\label{Theorem1.5}\,\
	Let $1<r\leq\infty$, $\rho>2$, $b\in {\rm BMO}(\mathbb{R}^n)$. Assume that $K(x,y)\in \mathcal{H}_r$ and satisfies $(\ref{(1.4)})$. Let $\mathcal{T}$ and $\mathcal{T}_{b}$ be given by $(\ref{(1.2)})$ and $(\ref{(1.3)})$, respectively. If $\mathcal{V}_\rho(\mathcal{T})$ is bounded on $L^{q_0}(\mathbb{R}^n,dx)$ for some $1<q_0<\infty$, then for every weight $\omega$ and every $0<\varepsilon\leq1$, all $\lambda>0$,
	\begin{align*}
		\omega(\{x\in\mathbb{R}^n:\mathcal{V}_{\rho}(\mathcal{T}_bf)(x)>\lambda\})&\leq C(n,r',q_0,\|\mathcal{V}_\rho(\mathcal{T})\|_{L^{q_0}\rightarrow L^{q_0}})\frac{1}{\varepsilon}\int_{\mathbb{R}^n}
		\psi\Big(\frac{\|b\|_{\rm BMO}|f(x)|}{\lambda}\Big)\\
		&\quad\times M_{L(\log L)^{4r'-3}(\log\log L)^{1+\varepsilon}}\omega(x)dx,
	\end{align*}
	and
	\begin{align*}
		\omega(\{x\in\mathbb{R}^n:\mathcal{V}_{\rho}(\mathcal{T}_bf)(x)>\lambda\})&\leq C(n,r',q_0,\|\mathcal{V}_\rho(\mathcal{T})\|_{L^{q_0}\rightarrow L^{q_0}})\frac{1}{\varepsilon}\int_{\mathbb{R}^n}
		\psi\Big(\frac{\|b\|_{\rm BMO}|f(x)|}{\lambda}\Big)\\
		&\quad\times M_{L(\log L)^{4r'-3+\epsilon}}\omega(x)dx,
	\end{align*}
	where $\psi(t)=t^{r'}[\log(e+t)]^{r'}$. Specially, if $\omega\in A_{\infty}$, then
	\begin{align*}
		\omega(\{x\in\mathbb{R}^n:\mathcal{V}_{\rho}(\mathcal{T}_bf)(x)>\lambda\})&\leq C(n,r',q_0,\|\mathcal{V}_\rho(\mathcal{T})\|_{L^{q_0}\rightarrow L^{q_0}})[\omega]_{A_\infty}^{4r'-3}\\
		&\quad\times[\log(e+[\omega]_{A_\infty})]\int_{\mathbb{R}^n}\psi
		\Big(\frac{\|b\|_{\rm BMO}|f(x)|}{\lambda}\Big)M\omega(x)dx.
	\end{align*}
	Moreover, if $\omega\in A_1$,
	\begin{align*}
		\omega(\{x\in\mathbb{R}^n:\mathcal{V}_{\rho}(\mathcal{T}_bf)(x)>\lambda\})&\leq C(n,r',q_0,\|\mathcal{V}_\rho(\mathcal{T})\|_{L^{q_0}\rightarrow L^{q_0}})[\omega]_{A_1}[\omega]_{A_\infty}^{4r'-3}\\
		&\quad\times[\log(e+[\omega]_{A_\infty})]\int_{\mathbb{R}^n}\psi\Big(\frac{\|b\|_{\rm BMO}|f(x)|}{\lambda}
		\Big)\omega(x)dx.
	\end{align*}
\end{theorem}

\begin{remark}
	When $r'=1$, the conclusions of Theorem $\ref{Theorem1.5}$ coincide with the case of commutator of Calder\'{o}n-Zygmund operator in \cite{IR, LOR1}. And the weak-type endpoint estimates for the variation of commutators are new, even in the un-weighted case. In addition, note that if $\Omega\in {\rm Lip}_\alpha(S^{n-1})$ with $0<\alpha\le 1$, then $K(x,y)=\Omega(x-y)/|x-y|^n\in \mathcal{H}_{\rm Dini}\subset \mathcal{H}_r$, $1\le r\le \infty$. Combing our Theorem $\ref{Theorem1.5}$ and  \cite[Theorem A]{CJRW2}, we obtain the following corollary.
	
\end{remark}

\begin{corollary}\label{Corollary1.6}\,\
	Let $\Omega$ be a homogeneous function of degree zero and have mean value zero on $S^{n-1}$ ( the unit sphere in $\mathbb{R}^n$).  Suppose that $\Omega\in {\rm Lip}_\alpha(S^{n-1})$ with $0<\alpha\le 1$, $b\in {\rm BMO}(\mathbb{R}^n)$, $\omega\in A_1$ and $\rho>2$. Then for all functions $f$ and all $\lambda> 0$,
	\begin{align*}
		\omega(\{x\in\mathbb{R}^n:\mathcal{V}_{\rho}(\mathcal{T}_{\Omega,b}f)(x)>\lambda\})&\leq C[\omega]_{A_1}[\omega]_{A_\infty}[\log(e+[\omega]_{A_\infty})]\\
		&\qquad\times\int_{\mathbb{R}^n}\Phi\Big(\frac{\|b\|_{\rm BMO}|f(x)|}{\lambda}
		\Big)\omega(x)dx,
	\end{align*}
	where $\Phi(t)=t\log(e+t)$, $\mathcal{T}_{\Omega,b}$ is defined as $\mathcal{T}_b$ in Definition $\ref{def1.2}$ with $K(x,y)=\Omega(x-y)/|x-y|^n$.
\end{corollary}

We organize the rest of the paper as follows. In Section 2, we will recall some related definitions and auxiliary lemmas. The proofs of Theorems $\ref{Theorem1.1}$ and $\ref{Theorem1.2}$ will be given in Section 3. In Section 4  we will prove Theorems $\ref{Theorem1.4}$ and $\ref{Theorem1.5}$. Finally, we will apply the sparse domination to present the local exponential decay estimates of variation operators in Section 5.

Through out the rest of our paper, we will denote positive constants by $C$, which may change at each occurrence. If $f\leq Cg$ and $f\lesssim g \lesssim f$, we denote $f\lesssim g$, $f\thicksim g$, respectively. We write the side length of $Q$ by $l_Q$.

\section{Preliminaries}
In this section, we recall some well known definitions and properties which will be used later.
\subsection{Weights}

A weight is a nonnegative and locally integrable function on $\mathbb{R}^n$. Given a pair of weights $\omega$ and $\sigma$, which satisfy
$$[\omega,\sigma]_{A_p}=\sup_{Q}\Big(\frac{1}{|Q|}\int_{Q}w(y)dy\Big)\Big(\frac{1}{|Q|}
\int_{Q}\sigma(y)dy\Big)^{p-1}<\infty,$$
and for $\omega\in A_{\infty}$,
$$[\omega]_{A_\infty}=\sup_{Q}\frac{1}{\omega(Q)}\int_{Q}M(\omega\chi_Q)(x)dx,$$
where the supremum is taken over all cubes $Q\subset \mathbb{R}^n$. Observe that if $\omega\in A_p$ and $\sigma=\omega^{1-p'}$, then $[\omega,\sigma]_{A_p}=[\omega]_{A_p}$. Recall that $A_\infty=\bigcup_{p\geq1}A_p$. The $A_\infty$ constant $[\omega]_{A_\infty}$ given above was shown in \cite{HP} to be the most suitable one and the following optimal reverse H\"{o}rder inequality was also obtained.
\begin{lemma}\rm(cf. \cite{HP})\label{lm 2.1}\,\
	Let $\omega\in A_\infty$, then for every cube $Q$,
	$$\Big(\frac{1}{|Q|}\int_{Q}\omega(x)^{r_\omega}dx\Big)^{1/r_\omega}\leq
	\frac{2}{|Q|}\int_Q\omega(x)dx,$$
	where $r_\omega=1+1/(\tau_n[\omega]_{A_\infty})$ with $\tau_n$ a dimensional constant independent $\omega$ and $Q$.
\end{lemma}

\subsection{Sparse family}

In this subsection, we will introduce a quite useful tool, which has been borrowed from \cite{LN}.

In the following, we call $\mathcal{D}(Q)$ the dyadic grid obtained by repeatedly subdividing $Q$ and its descendants in $2^n$ cubes with the same side length.
\begin{definition}
	A family of cubes is said to be a dyadic lattice $\mathcal{D}$ if it satisfies the following properties:\\
	$(1)$ if $Q\in\mathcal{D}$, then every descendant of $Q$ is also in $\mathcal{D}$;\\
	$(2)$ for every two cubes $Q_1,Q_2\in\mathcal{D}$, we can find a common ancestor $Q\in\mathcal{D}$ such that $Q_1,Q_2\in\mathcal{D}(Q)$;\\
	$(3)$ for each compact set $K\subseteq\mathbb{R}^n$, we can find a cube $Q\in\mathcal{D}$ such that $K\subseteq Q$.
\end{definition}

The following lemma is called the Three Lattice Theorem, which will play a key role in our proofs.
\begin{lemma}\rm(cf. \cite{LN})\label{lm 2.3}\,\
	Given a dyadic lattice $\mathcal{D}$, there exist $3^n$ dyadic lattices $\mathcal{D}_1,\dots,\mathcal{D}_{3^n}$ such that
	$$\{3Q:Q\in\mathcal{D}\}=\bigcup_{j=1}^{3^n}\mathcal{D}_j$$
	and for each cube $Q\in\mathcal{D}$ we can find a cube $R_Q$ in each $\mathcal{D}_j$ such that $Q\subseteq R_Q$ and $3l_Q=l_{R_Q}$.
\end{lemma}
\begin{remark}\label{Remark 2.4}\,\
	Fix a dyadic lattice $\mathcal{D}$. For any cube $Q\subset\mathbb{R}^n$, we can always find a cube $Q'\in\mathcal{D}$ such that $l_Q/2<l_{Q'}\leq l_Q$ and $Q\subset3Q'$. By the above lemma, for some $j\in\{1,\dots,3^n\}$, it is easy to see that $3Q'=P\in\mathcal{D}_j$. Hence, for each cube $Q\subset\mathbb{R}^n$, we can find a cube $P\in\mathcal{D}_j$ that satisfies $Q\subset P$ and $l_P\leq3l_Q$.
\end{remark}

By the definition of dyadic lattice, we can give the definition of sparse family.
\begin{definition}
	Let $\mathcal{D}$ be a dyadic lattice. $\mathcal{S}\subset\mathcal{D}$ is a $\eta$-sparse family with $\eta\in(0,1)$ if for every cube $Q\in\mathcal{S}$, we can find a measurable subset $E_Q\subset Q$ such that $\eta|Q|\leq|E_Q|$, where all the $E_Q$ are pairwise disjoint.
\end{definition}
Let $r>0$ and $\mathcal{S}$ be a $\eta$-sparse family, we define the sparse operator as
\begin{equation*}
	\mathcal{A}_{r,\mathcal{S}}f(x)=\Big(\sum_{Q\in\mathcal{S}}\Big(\frac{1}{|Q|}\int_{Q}|f(y)|dy
	\Big)^r\chi_Q(x)\Big)^{1/r}.
\end{equation*}
Let $\omega,\sigma$ be a pair of weights, for $1<p<\infty$, $r>0$. In \cite{HL}, Hyt\"{o}nen et al. proved that
\begin{equation}\label{eq 2.1}\,\
	\|\mathcal{A}_{r,\mathcal{S}}(\sigma f)\|_{L^p(\omega)}\lesssim [\omega,\sigma]_{A_{p}}^{1/p}\Big([\omega]_{A_\infty}^{(1/r-1/p)_+}+
	[\sigma]_{A_\infty}^{1/p}\Big)\|f\|_{L^p(\sigma)},
\end{equation}
\begin{equation}\label{eq 2.2}\,\
	\|\mathcal{A}_{r,\mathcal{S}}(\sigma f)\|_{L^{p,\infty}(\omega)}\lesssim [\omega,\sigma]_{A_{p}}^{1/p}[\omega]_{A_\infty}^{(1/r-1/p)_+}
	\|f\|_{L^p(\sigma)}\quad p\neq r,
\end{equation}
where
\begin{equation*}
	(\alpha)_{+}=
	\begin{cases}
		\alpha\quad$if$~\alpha>0,\\
		0\quad$otherwise$.\\
	\end{cases}
\end{equation*}

\subsection{Young function and Orlicz maximal operators}

In this subsection, we recall some fundamental facts about Young functions, Orlicz local averages, for more details, see \cite{RR}.

We say that a function $A:[0,\infty)\rightarrow[0,\infty)$ is a Young function if $A$ is a continuous, convex increasing function that satisfies $A(0)=0$ and $\lim_{t\rightarrow\infty}A(t)=\infty$.
The $A$-norm of $f$ over $Q$ is defined as
$$\|f\|_{A(\mu),Q}:=\inf\Big\{\lambda>0:\frac{1}{\mu(Q)}\int_QA
\Big(\frac{|f(x)|}{\lambda}\Big)d\mu\leq1\Big\}.$$
We'll denote $\|f\|_{A,Q}$ if $\omega$ is the Lebesgue measure and write $\|f\|_{A(\omega),Q}$ if $\mu=\omega dx$ is an absolutely continuous measure with respect to the Lebesgue measure. We then define the Orlicz maximal operator $M_Af(x)$ in a natural way by
$$M_Af(x):=\sup_{Q\ni x}\|f\|_{A,Q}.$$
For every Young function $A$, we can define its complementary function $\overline{A}$ by
$$\overline{A}(t)=\sup_{s>0}\{st-A(s)\}.$$
There are some interesting properties such as a generalized H\"{o}lder inequality
\begin{equation}\label{eq 2.3}\,\
	\frac{1}{\mu(Q)}\int_Q|f(x)g(x)|d\mu(x)\leq2\|f\|_{A(\mu),Q}\|g\|_{\bar{A}(\mu),Q}.
\end{equation}
Now we present some particular cases of maximal operators.\\
\quad$\bullet$ If $A(t)=t^r$ with $r>1$, then $M_A=M_r$.\\
\quad$\bullet$ $M_A=M_{L\log L^\alpha}$ with $\alpha>0$ given by the function $A(t)=t(\log(e+t))^\alpha$. Note that for any $\alpha>0$, $M\lesssim M_A\lesssim M_r$ for every $1<r<\infty$.\\
\quad$\bullet$ If $A(t)=t\log(e+t)^\alpha\log(e+\log(e+t))^{\beta}$ with $\alpha,\beta>0$, then we denote $M_A=M_{L(\log L)^\alpha(\log\log L)^\beta}$.

Finally, to prove Theorem $\ref{Theorem1.5}$, we will use the following lemma, which was proved essentially in \cite{IR}.

\begin{lemma}\label{lm 2.6}\,\
	Let $b\in {\rm BMO}$ and $\psi_0(t)=t^{r'}$ for $1\leq r'<\infty$. Let $\psi$ be a Young function such that $\psi_{1}^{-1}(t)\bar{\psi}_{0}^{-1}(t)\bar{C}_{1}^{-1}(t)\lesssim t$ with $\bar{C}_{1}(t)=e^{t}$ for $t\geq1$. Assume that $\psi(xy)\lesssim\psi(x)\psi(y)$, and $\beta_{n}$ is the constant such that $e^{\frac{(3/2)^{k-1}}{2^ne}-1}\geq\max(e^2,4^k)$ for $k>\beta_{n}$. Then for every weight $\omega$, and Young functions $\varphi_0,\varphi_1$,
	$$\omega(\{x\in\mathbb{R}^n:\mathcal{A}_{\mathcal{S},b}(f)(x)>\lambda\})\leq c\sum_{j=0}^{1}\Big(\kappa_{\varphi_j}\int_{\mathbb{R}^n}\psi_j\Big({\frac{\|b\|_{\rm BMO}|f(x)|}{\lambda}}\Big)
	M_{\Phi_{1-j}\circ\varphi_j}\omega(x)dx\Big),$$
	where $\mathcal{A}_{\mathcal{S},b}(f)(x)=\sum_{R\in \mathcal{S}}\Big\{|b(x)-b_R|\langle|f|^{r'}\rangle_{R}^{\frac{1}{r'}}+
	\langle|f(b-b_R)|^{r'}\rangle_{R}^{\frac{1}{r'}}\Big\}\chi_R(x)$, $\Phi_{j}(t)=t\log(e+t)^{j}$, $j=0,1$, and
	\begin{equation*}
		\kappa_{\varphi_j}=
		\begin{cases}
			\sum_{k=1}^{\beta_{n}}\frac{4^{k(r'-1)}\varphi_{0}^{-1}\circ\Phi_{1}^{-1}(1/\alpha_k)}
			{\Phi_{1}^{-1}(1/\alpha_k)}+c_n\int_{1}^{\infty}\frac{\varphi_{0}^{-1}\circ\Phi_{1}^{-1}(t)
				\psi_0((\log(e+t))^4)}{t^2(\log(e+t))^4},~j=0\\
			\int_{1}^{\infty}\frac{\varphi_{1}^{-1}(t)\psi((\log(e+t))^2)}{t^2(\log(e+t))^3}dt,~j=1\\
		\end{cases}
	\end{equation*}
	with $\alpha_k=\min(1,e^{-\frac{(3/2)^{k}}{2^ne}+1})$.
\end{lemma}
\begin{remark}\label{Remark 2.7}\,\
	Lemma $\ref{lm 2.6}$ is a special case of Theorem $2.7$ in \cite{IR} with a carefully
	checking in the constants. We omit the details.
\end{remark}

\section{Proofs of Theorems $\ref{Theorem1.1}$ and $\ref{Theorem1.2}$}\label{section 3}

This section is devoted to Theorems $\ref{Theorem1.1}$ and $\ref{Theorem1.2}$, which will be based on a pointwise estimate of the grand maximal operators associated with the variation operators.
At first, we define the grand maximal truncated operators $\mathcal{M}_{\mathcal{V}_{\rho}(\mathcal{T})}$ by
$$\mathcal{M}_{\mathcal{V}_{\rho}(\mathcal{T})}f(x)=\sup_{Q\ni x}\esss_{\xi\in Q}\mathcal{V}_{\rho}(\mathcal{T}f\chi_{\mathbb{R}^n\backslash 3Q})(\xi).$$
Given a cube $Q_0$, we also consider a local version of $\mathcal{M}_{\mathcal{V}_{\rho}(\mathcal{T}),Q_0}$ by
\begin{equation*}
	\mathcal{M}_{\mathcal{V}_\rho(\mathcal{T}),Q_0}f(x):=
	\begin{cases}
		\sup_{Q\ni x,Q\subset Q_0}\esss_{\xi\in Q}\mathcal{V}_{\rho}(\mathcal{T}f\chi_{3Q_0\backslash 3Q})(\xi), ~x\in Q_0.\\
		0,~$otherwise$.\\
	\end{cases}
\end{equation*}
To prove our results, we need to fix some notations as in \cite{GT}. Set $\Theta=\{\beta:\beta=\{\epsilon_i\},\epsilon_i\in\mathbb{R},\epsilon_i\downarrow0\}$. We consider the set $\mathbb{N}\times\Theta$ and denote the mixed norm space of two variables functions $g(i,\beta)$ by $F_\rho$ such that
$$\|g\|_{F_\rho}\equiv\sup_{\beta}\Big(\sum_{i}|g(i,\beta)|^\rho\Big)^{{1}/{\rho}}
<\infty.$$
We consider the $F_\rho$-valued operator $V(\mathcal{T}):f\rightarrow V(\mathcal{T})f$ defined by
$$V(\mathcal{T})f(x):=\{T_{\epsilon_{i+1}}f(x)-T_{\epsilon_i}f(x)\}_{\beta=\{\epsilon_i\}\in \Theta}.$$
Then
$$\mathcal{V}_{\rho}(\mathcal{T}f)(x)=\|V(\mathcal{T})f(x)\|_{F_{\rho}}.$$

\begin{lemma}\label{lm 3.1}\,\
	Let $\mathcal{T}=\{T_\varepsilon\}_{\varepsilon>0}$ be given by $(\ref{(1.2)})$. Suppose that $\mathcal{V}_{\rho}(\mathcal{T})$ is bounded on $L^{q_0}(\mathbb{R}^n)$ for some $1<q_0<\infty$. Then for a.e. $x\in Q_0$, $f\in C_{c}^{\infty}(\mathbb{R}^n)$,
	\begin{equation}\label{eq 3.1}\,\
		\mathcal{V}_{\rho}(\mathcal{T}f\chi_{3Q_0})(x)\leq c_{n,q_0}\|{\mathcal{V}}_{\rho}(\mathcal{T})\|_{L^{q_0}\rightarrow L^{q_0}}|f(x)|+\mathcal{M}_{\mathcal{V}_{\rho}(\mathcal{T}),Q_0}f(x).
	\end{equation}
\end{lemma}
\begin{proof}
	For $x\in$ int $Q_0$, let $x$ be a point of approximate continuity of $\mathcal{V}_{\rho}(\mathcal{T}f\chi_{3Q_0})(x)$ (see \cite{EG}). For any $\varepsilon>0$, set
	$$E_{r}(x)=\{y\in B(x,r):|\mathcal{V}_{\rho}(\mathcal{T}f\chi_{3Q_0})(y)-\mathcal{V}_{\rho}
	(\mathcal{T}f\chi_{3Q_0})(x)|<\varepsilon\},$$ where $B(x,r)$ is the ball centered at $x$ of radius $r$. Then we immediately deduce that $\lim_{r\rightarrow0}|E_r(x)|/|B(x,r)|=1$. We denote the smallest cube centered at $x$ and containing $B(x,r)$ by $Q(x,r)$. Let $r>0$ be small enough so that $Q(x,r)\subset Q_0$. Hence, for a.e. $y\in E_r(x)$,
	\begin{equation*}
		\begin{split}
			\mathcal{V}_{\rho}(\mathcal{T}f\chi_{3Q_0})(x)&<
			\mathcal{V}_{\rho}(\mathcal{T}f\chi_{3Q_0})(y)+\varepsilon\\
			&\leq \mathcal{V}_{\rho}(\mathcal{T}f\chi_{3Q(x,r)})(y)+\mathcal{M}_{\mathcal{V}_
				{\rho}(\mathcal{T}),Q_0}f(x)+\varepsilon.
		\end{split}
	\end{equation*}
	Applying the $L^{q_0}$-boundedness of $\mathcal{V}_\rho(\mathcal{T})$, we have
	\begin{equation*}
		\begin{split}
			\mathcal{V}_{\rho}(\mathcal{T}f\chi_{3Q_0})(x)&\leq\Big(\frac{1}{|E_r(x)|}\int_{E_r(x)
			}\mathcal{V}_{\rho}(\mathcal{T}f\chi_{3Q(x,r)})(y)^{q_0} dy\Big)^{\frac{1}{q_0}}+
			\mathcal{M}_{\mathcal{V}_{\rho}(\mathcal{T}),Q_0}f(x)+\varepsilon\\
			&\leq\|\mathcal{V}_{\rho}(\mathcal{T})\|_{L^{q_0}\rightarrow L^{q_0}}\Big(\frac{1}{|E_r(x)|}\int_{3Q(x,r)}|f(y)|^{q_0}dy\Big)^{\frac{1}{q_0}}+
			\mathcal{M}_{\mathcal{V}_{\rho}(\mathcal{T}),Q_0}f(x)+\varepsilon.
		\end{split}
	\end{equation*}
	By additionally assuming that $x$ is a Lebesgue point of $f$, we get the results by letting $r,\varepsilon\rightarrow 0$.
\end{proof}

\begin{lemma}\label{lm 3.2}\,\
	Let $\mathcal{T}=\{T_\varepsilon\}_{\varepsilon>0}$ be given by $(\ref{(1.2)})$. Assume that $K(x,y)$ satisfies $(\ref{(1.4)})$ and $K\in \mathcal{H}_r$ for $1<r\leq\infty$. If $\mathcal{V}_{\rho}(\mathcal{T})$ is bounded on $L^{q_0}(\mathbb{R}^n)$ for some $1<q_0<\infty$, then
	\begin{align*}
		\|\mathcal{M}_{\mathcal{V}_{\rho}(\mathcal{T}),Q_0}\|_{L^{r'}\rightarrow L^{r',\infty}}&\leq C(n,r',q_0,\|\mathcal{V}_\rho(\mathcal{T})\|_{L^{q_0}\rightarrow L^{q_0}}),
	\end{align*}
	where $1/r+1/r'=1$.
\end{lemma}

\begin{proof}
	We first consider the case $1<r<\infty$. For any $x\in Q_0$, $Q\ni x$ with $Q\subset Q_0$ and $\xi\in Q$, we denote $B_x:=B(x,9n l_Q)$ and $\widetilde{B_x}:=B(x,3\sqrt{n}l_{Q_0})$. Then $3Q\subset B_x$ and $3Q_0\subset \widetilde{B_x}$. We have
	\begin{equation*}
		\begin{split}
			\mathcal{V}_{\rho}(\mathcal{T}f\chi_{3Q_0\backslash 3Q})(\xi)&\leq\mathcal{V}_{\rho}(\mathcal{T}f\chi_{(B_x\cap3Q_0)\backslash 3Q})(\xi)+\mathcal{V}_{\rho}(\mathcal{T}f\chi_{3Q_0\backslash B_x})(x)\\
			&\quad+|\mathcal{V}_
			{\rho}(\mathcal{T}f\chi_{3Q_0\backslash B_x})(\xi)-\mathcal{V}_
			{\rho}(\mathcal{T}f\chi_{3Q_0\backslash B_x})(x)|\\
			&=:I+II+III.
		\end{split}
	\end{equation*}
	Now we estimate $I$, $II$ and $III$, respectively. Note that $$\|\{\chi_{\{\varepsilon_{i+1}<|\xi-y|\leq\varepsilon_i\}}\}_{\beta=\{\varepsilon_i\}\in\Theta}\|
	_{F_\rho}\leq1.$$
	By the Minkowski inequality and $(\ref{(1.4)})$, one can see that
	\begin{equation}\label{eq 3.4}
		\begin{split}
			I&\leq\int_{\mathbb{R}^n}\|\{\chi_{\{\varepsilon_{i+1}<|\xi-y|\leq\varepsilon_i\}}\}
			_{\beta=\{\varepsilon_i\}\in\Theta}\|_{F_\rho}|f(y)\chi_{(B_x\cap3Q_0)\backslash 3Q}(y)||K(\xi,y)|dy\\
			&\leq\int_{(B_x\cap3Q_0)\backslash 3Q}\frac{|f(y)|}{|\xi-y|^n}dy\leq C_n M_{r'}f(x).
		\end{split}
	\end{equation}
	
	Similarly, by the definition and sublinearity of $\mathcal{V}_\rho(\mathcal{T})$, we have
	\begin{equation}\label{eq 3.5}
		II\leq2\mathcal{V}_\rho(\mathcal{T}f)(x)+\mathcal{V}_\rho(\mathcal{T}f\chi_{\widetilde{B_x}
			\backslash 3Q_0})(x)\leq2\mathcal{V}_\rho(\mathcal{T}f)(x)+C_n M_{r'}f(x).
	\end{equation}
	
	For the term $III$, we can write
	\begin{align*}
		&III\leq\Big\|\Big\{\int_{\varepsilon_{i+1}<|\xi-y|\leq\varepsilon_i}K(\xi,y)f(y)\chi_
		{3Q_0\backslash B_x}(y)dy\\
		&\quad\quad-\int_{\varepsilon_{i+1}<|x-y|\leq\varepsilon_i}K(x,y)f(y)
		\chi_{3Q_0\backslash B_x}(y)dy\Big\}_{\beta=\{\epsilon_i\}\in\Theta}\Big\|_{F_\rho}\\
		&\quad\leq\Big\|\Big\{\int_{\varepsilon_{i+1}<|\xi-y|\leq\varepsilon_i}(K(\xi,y)-K(x,y))f(y)
		\chi_{3Q_0\backslash B_x}(y)dy\Big\}_{\beta=\{\epsilon_i\}\in\Theta}\Big\|_{F_\rho}\\
		&\quad\quad+\Big\|\Big\{\int_{\mathbb{R}^n}(\chi_{\{\varepsilon_{i+1}<|\xi-y|\leq\varepsilon_i\}}
		(y)-\chi_{\{\varepsilon_{i+1}<|x-y|\leq\varepsilon_i\}}(y))K(x,y)f(y)
		\chi_{3Q_0\backslash B_x}(y)dy\Big\}\Big\|_{F_\rho}\\
		&\quad=:I_1+I_2.
	\end{align*}
	Since $\|\{\chi_{\{\varepsilon_{i+1}<|\xi-y|\leq\varepsilon_i\}}\}_{\beta=\{\varepsilon_i\}\in\Theta}\|
	_{F_\rho}\leq1$, by the Minkowski inequality and the H\"{o}rmander condition,
	\begin{align*}
		I_1&\leq\int_{\mathbb{R}^n}\|\{\chi_{\{\varepsilon_{i+1}<|\xi-y|\leq\varepsilon_i\}}\}
		_{\beta=\{\varepsilon_i\}\in\Theta}\|_{F_\rho}|f(y)\chi_{\mathbb{R}^n\backslash 3Q}(y)||K(\xi,y)-K(x,y)|dy\\
		&\leq\sum_{k=1}^{\infty}2^{kn}(3l_Q)^n\frac{1}{|2^k3Q|}\int_{2^k3Q\backslash 2^{k-1}3Q}|f(y)||K(\xi,y)-K(x,y)|dy\\
		&\leq\sum_{k=1}^{\infty}2^{kn}(3l_Q)^n\Big(\frac{1}{|2^k3Q|}\int_{2^k3Q\backslash 2^{k-1}3Q}|K(\xi,y)-K(x,y)|^r dy\Big)^{1/r}\\
		&\quad\Big(\frac{1}{|2^k3Q|}\int_{2^k3Q}|f(y)|^{r'}dy\Big)^{1/r'}\\
		&\leq C M_{r'}(f)(x).
	\end{align*}
	Next, we deal with the term $I_2$. As we can see, the integral
	$$\int_{\mathbb{R}^n}|\chi_{\{\varepsilon_{i+1}<|\xi-y|\leq\varepsilon_i\}}(y)-
	\chi_{\{\varepsilon_{i+1}<|x-y|\leq\varepsilon_i\}}(y)||K(x,y)f(y)|
	\chi_{3Q_0\backslash B_x}(y)dy$$ will be non-zero if either
	$$\chi_{\{\varepsilon_{i+1}<|\xi-y|\leq\varepsilon_i\}}(y)=1 \quad and \quad \chi_{\{\varepsilon_{i+1}<|x-y|\leq\varepsilon_i\}}(y)=0$$
	or viceversa. Therefore, we need to consider the following four cases:\\
	(i)\quad $\varepsilon_{i+1}<|\xi-y|\leq\varepsilon_i$ and $|x-y|\leq\varepsilon_{i+1}$,\\
	(ii)\quad $\varepsilon_{i+1}<|\xi-y|\leq\varepsilon_i$ and $|x-y|>\varepsilon_{i}$,\\
	(iii)\quad $\varepsilon_{i+1}<|x-y|\leq\varepsilon_i$ and $|\xi-y|\leq\varepsilon_{i+1}$,\\
	(iv)\quad $\varepsilon_{i+1}<|x-y|\leq\varepsilon_i$ and $|\xi-y|>\varepsilon_{i}$.\\
	In case (i) we have $\varepsilon_{i+1}<|\xi-y|\leq|x-\xi|+|x-y|\leq  \sqrt{n}l_Q+\varepsilon_{i+1}$. Other cases are similar, and we conclude that: in case (ii), $\varepsilon_i<|x-y|\leq \sqrt{n}l_Q+\varepsilon_i$; in case (iii), $\varepsilon_{i+1}<|x-y|\leq \sqrt{n}l_Q+\varepsilon_{i+1}$; in case (iv), $\varepsilon_{i}<|\xi-y|\leq \sqrt{n}l_Q+\varepsilon_{i}$.
	
	Therefore, using $|K(x,y)|\leq C/|x-y|^n$, we obtain
	\begin{align*}
		&\int_{\mathbb{R}^n}|\chi_{\{\varepsilon_{i+1}<|\xi-y|\leq\varepsilon_i\}}(y)-\chi_
		{\{\varepsilon_{i+1}<|x-y|\leq\varepsilon_i\}}(y)||K(x,y)||f(y)|\chi_{3Q_0\backslash B_x}(y)dy\\
		&\quad\leq C\int_{\mathbb{R}^n}\chi_{\{\varepsilon_{i+1}<|\xi-y|\leq\varepsilon_i\}}(y)
		\chi_{\{\varepsilon_{i+1}<|\xi-y|\leq\varepsilon_{i+1}+\sqrt{n}l_Q\}}(y)\frac{1}{|x-y|^n}
		|f(y)|\chi_{3Q_0\backslash B_x}(y)dy\\
		&\qquad+C\int_{\mathbb{R}^n}\chi_{\{\varepsilon_{i+1}<|\xi-y|\leq\varepsilon_i\}}(y)
		\chi_{\{\varepsilon_{i}<|x-y|\leq\varepsilon_{i}+\text {diam}Q\}}(y)\frac{1}{|x-y|^n}|
		f(y)|\chi_{3Q_0\backslash B_x}(y)dy\\
		&\qquad+C\int_{\mathbb{R}^n}\chi_{\{\varepsilon_{i+1}<|x-y|\leq\varepsilon_i\}}(y)
		\chi_{\{\varepsilon_{i+1}<|x-y|\leq\varepsilon_{i+1}+\sqrt{n}l_Q\}}(y)\frac{1}{|x-y|^n}
		|f(y)|\chi_{3Q_0\backslash B_x}(y)dy\\
		&\qquad+C\int_{\mathbb{R}^n}\chi_{\{\varepsilon_{i+1}<|x-y|\leq\varepsilon_i\}}(y)
		\chi_{\{\varepsilon_{i}<|\xi-y|\leq\varepsilon_{i}+\sqrt{n}l_Q\}}(y)\frac{1}
		{|x-y|^n}|f(y)|\chi_{3Q_0\backslash B_x}(y)dy\\
		&\quad=:I_{21}+I_{22}+I_{23}+I_{24}.
	\end{align*}
	Observe that $\sqrt{n}l_Q\geq \varepsilon_{i+1}$ implies $I_{21}=I_{23}=0$, and $I_{22}=I_{24}=0$ if $\sqrt{n}l_Q\geq\varepsilon_i$. Using $c_0|x-y|\leq|\xi-y|\leq c_1|x-y|$ with constants $c_0,c_1>0$, we may assume that $c_0<1$ and $c_1>1$, otherwise, $I_{22}=I_{24}=0$. For $1<t<\min(r',2)$,  by the H\"{o}lder inequality, we get
	\begin{equation*}
		\begin{split}
			I_{21}&\leq C\left(\int_{\mathbb{R}^n}\chi_{\{\varepsilon_{i+1}<
				|\xi-y|\leq\varepsilon_i\}}(y)|f(y)|^{t}\chi_{3Q_0\backslash B_x}(y)\frac{1}
			{|x-y|^{nt}}dy\right)^{1/t}\\
			&\quad\times[(\sqrt{n}l_Q+\varepsilon_{i+1})^n-(\varepsilon_{i+1})^n]^{1/t'},
		\end{split}
	\end{equation*}
	
	\begin{equation*}
		\begin{split}
			I_{22}&\leq C\left(\int_{\mathbb{R}^n}\chi_{\{\max\{\varepsilon_{i+1}, 2\varepsilon_i/3\}<
				|\xi-y|\leq\varepsilon_i\}}(y)|f(y)|^{t}\chi_{3Q_0\backslash B_x}(y)\frac{1}
			{|x-y|^{nt}}dy\right)^{1/t}\\
			&\quad\times[\sqrt{n}l_Q+\varepsilon_{i})^n-(\varepsilon_{i})^n]^{1/t'},
		\end{split}
	\end{equation*}
	
	\begin{equation*}
		\begin{split}
			I_{23}&\leq C\left(\int_{\mathbb{R}^n}\chi_{\{\varepsilon_{i+1}<
				|x-y|\leq\varepsilon_i\}}(y)|f(y)|^{t}\chi_{3Q_0\backslash B_x}(y)\frac{1}
			{|x-y|^{nt}}dy\right)^{1/t}\\
			&\quad\times[(\sqrt{n}l_Q+\varepsilon_{i+1})^n-(\varepsilon_{i+1})^n]^{1/t'},
		\end{split}
	\end{equation*}
	\begin{equation*}
		\begin{split}
			I_{24}&\leq C\left(\int_{\mathbb{R}^n}\chi_{\{\max\{\varepsilon_{i+1},3\varepsilon_i/4\}<
				|x-y|\leq\varepsilon_i\}}(y)|f(y)|^{t}\chi_{3Q_0\backslash B_x}(y)\frac{1}
			{|x-y|^{nt}}dy\right)^{1/t}\\
			&\quad\times[(\sqrt{n}l_Q+\varepsilon_{i})^n-(\varepsilon_{i})^n]^{1/t'}.
		\end{split}
	\end{equation*}
	
	Note that $\sqrt{n}l_Q<\varepsilon_{i+1}$, we have
	$$(\sqrt{n}l_Q+\varepsilon_{i+1})^n-\varepsilon_{i+1}^n\leq C(\max\{\sqrt{n}l_Q,\varepsilon_{i+1}\})^{n-1}\sqrt{n}l_Q\leq C\varepsilon_{i+1}^{n-1}\sqrt{n}l_Q.$$
	
	Then
	\begin{equation*}
		\begin{split}
			I_{21}&\leq C_{n,r'} \frac{[(\sqrt{n}l_Q+\varepsilon_{i+1})^n-(\varepsilon_{i+1})^n]^{1/t'}}{(\varepsilon_{i+1})
				^{(n-1)/t'}}\\
			&\quad\times\left(\int_{\mathbb{R}^n}\chi_{\{\varepsilon_{i+1}<|\xi-y|\leq\varepsilon_i\}}(y)
			\frac{|f(y)|^{t}\chi_{3Q_0\backslash B_x}(y)}{|x-y|^{n+t-1}}dy\right)^{1/t}\\
			&\leq C_{n,r'}(\sqrt{n}l_{Q})^{\frac{1}{t'}}
			\left(\int_{\mathbb{R}^n}\chi_{\{\varepsilon_{i+1}<|\xi-y|\leq\varepsilon_i\}}
			(y)\frac{|f(y)|^{t}\chi_{3Q_0\backslash B_x}(y)}{|x-y|^{n+t-1}}dy\right)^{1/t}.
		\end{split}
	\end{equation*}
	Similarly,
	$$I_{22}\leq C_{n,r'}(\sqrt{n}l_{Q})^{\frac{1}{t'}}
	\left(\int_{\mathbb{R}^n}\chi_{\{\varepsilon_{i+1}<|\xi-y|\leq\varepsilon_i\}}
	(y)\frac{|f(y)|^{t}\chi_{3Q_0\backslash B_x}(y)}{|x-y|^{n+t-1}}dy\right)^{1/t}.$$
	$$I_{23}\leq C_{n,r'} (\sqrt{n}l_
	{Q})^{\frac{1}{t'}}\left(\int_{\mathbb{R}^n}\chi_{\{\varepsilon_{i+1}<|x-y|\leq\varepsilon_i\}
	}(y)\frac{|f(y)|^{t}\chi_{3Q_0\backslash B_x}(y)}{|x-y|^{n+t-1}}dy\right)^{1/t}.$$
	$$I_{24}\leq C_{n,r'} (\sqrt{n}l_
	{Q})^{\frac{1}{t'}}\left(\int_{\mathbb{R}^n}\chi_{\{\varepsilon_{i+1}<|x-y|\leq\varepsilon_i\}
	}(y)\frac{|f(y)|^{t}\chi_{3Q_0\backslash B_x}(y)}{|x-y|^{n+t-1}}dy\right)^{1/t}.$$
	Consequently,
	\begin{equation*}
		\begin{split}
			I_{2}&\leq C_{n,r'} (\sqrt{n}l_
			{Q})^{\frac{1}{t'}}\Big\|\Big\{\Big(\int_{\mathbb{R}^n}\chi_{\{\varepsilon_{i+1}<|\xi-y|\leq
				\varepsilon_i\}}(y)\frac{|f(y)|^{t}\chi_{3Q_0\backslash B_x}(y)}{|x-y|^{n+t-1}}dy\Big)^{1/t}
			\Big\}_{\beta=\{\varepsilon_i\}\in\Theta}\Big\|_{F_\rho}\\
			&\quad+C_{n,r'} (\sqrt{n}l_
			{Q})^{\frac{1}{t'}}\Big\|\Big\{\Big(\int_{\mathbb{R}^n}\chi_{\{\varepsilon_{i+1}
				<|x-y|<\varepsilon_i\}}(y)\frac{|f(y)|^{t}\chi_{\mathbb{R}^n\backslash 3Q}(y)}{|x-y|^{n+t-1}}dy\Big)
			^{1/t}\Big\}_{\beta=\{\varepsilon_i\}\in\Theta}\Big\|_{F_\rho}\\
			&=:D_{1}+D_{2}.
		\end{split}
	\end{equation*}
	A direct computation shows that
	\begin{equation*}
		\begin{split}
			&\Big\|\Big\{\Big(\int_{\mathbb{R}^n}\chi_{\{\varepsilon_{i+1}<|\xi-y|\leq\varepsilon_i\}}
			(y)\frac{\chi_{3Q_0\backslash B_x}(y)
			}{|x-y|^{n+t-1}}|f(y)|^{t} dy\Big)^{1/t}\Big\}_{\beta=\{\varepsilon_i\}
				\in\Theta}\Big\|_{F_\rho}\\
			&\quad= \sup_{\varepsilon_i}\Big[\sum_i\Big(\int_{\mathbb{R}^n}\chi_{\{\varepsilon_{i+1}<|\xi-y|
				\leq\varepsilon_i\}}(y)\frac{\chi_{3Q_0\backslash B_x}(y)}{|x-y|^{n+t-1}}|f(y)|^{t} dy\Big)^{\rho/t}\Big]^{1/\rho}\\
			&\quad\leq \Big(\int_{3Q_0\backslash B_x}
			\frac{|f(y)|^{t}}{|x-y|^{n+t-1}}dy\Big)^{1/t}\leq C_{n,r'}(\sqrt{n}l_Q)^{-\frac{1}{t'}}M_{r'}(f)(x).
		\end{split}
	\end{equation*}
	This implies that $D_{1}\leq C_{n,r'}M_{r'}(f)(x)$.
	
	Similarly, we can also deduce that $D_{2}\leq C_{n,r'}M_{r'}(f)(x).$
	
	Hence, $I_{2}\leq C_{n,r'} M_{r'}(f)(x),$ which, together with the estimates of $I_{1}$, implies that
	\begin{equation}\label{eq 3.6}
		III=|\mathcal{V}_{\rho}(\mathcal{T}f\chi_{3Q_0\backslash B_x})(\xi)
		-\mathcal{V}_{\rho}(\mathcal{T}f\chi_{3Q_0\backslash B_x})(x)|\leq C_{n,r'}M_{r'}f(x).
	\end{equation}
	It was proved in \cite{ZW} that $\mathcal{V}_\rho(\mathcal{T})$ are strong type $(p,p)$ with $p>r'$ and weak type $(1,1)$, then using the interpolation theorem, we have
	$$\|\mathcal{V}_\rho(\mathcal{T})\|_{L^{r'}\rightarrow L^{r'}}\leq C(n,r',q_0,\|\mathcal{V}_\rho(\mathcal{T})\|_{L^{q_0}\rightarrow L^{q_0}}).$$
	Then by $(\ref{eq 3.4})$-$(\ref{eq 3.6})$, we get the desired result for the case of $1<r<\infty$.
	
	Now we turn to the case of $r=\infty$ by employing the idea in \cite{DuLY}. For any $f\in L^1(\mathbb{R}^n)$ and $\lambda>0$, applying the Calder\'{o}n-Zygmund decomposition to $f$ at height $\lambda$. We obtain $f=g+b$ such that\\
	($P_1$)\:$|g(x)|\leq2^n\lambda$ for a.e. $x\in\mathbb{R}^n$ and $\|g\|_{L^1(\mathbb{R}^n)}\leq\|f\|_{L^1(\mathbb{R}^n)}$;\\
	($P_2$)\:$b=\sum_jb_j$, $\text {supp}~ (b_j)\subset Q_j$ and $\{Q_j\}\subset\mathcal{D}(\mathbb{R}^n)$ is a pairwise disjoint family of cubes, where $\mathcal{D}(\mathbb{R}^n)$ is the family of dyadic cubes in $\mathbb{R}^n$;\\
	($P_3$)\:for every $j$, $\int_{\mathbb{R}^n}b_j(x)dx=0$ and $\|b_j\|_{L^1(\mathbb{R}^n)}\leq2^{n+1}\lambda|Q_j|$;\\
	($P_4$)\:$\sum_j|Q_j|\leq\|f\|_{L^1(\mathbb{R}^n)}/\lambda$.\\
	By the sublinearity of $\mathcal{M}_{\mathcal{V}_\rho(\mathcal{T}),Q_0}$, we have
	\begin{align*}
		&|\{x\in\mathbb{R}^n:\mathcal{M}_{\mathcal{V}_\rho(\mathcal{T}),Q_0}f(x)>\lambda\}|\\
		&\quad\leq|\{x\in\mathbb{R}^n:\mathcal{M}_{\mathcal{V}_\rho(\mathcal{T}),Q_0}g(x)>\lambda/2\}|+
		|\{x\in\mathbb{R}^n:\mathcal{M}_{\mathcal{V}_\rho(\mathcal{T}),Q_0}b(x)>\lambda/2\}|.
	\end{align*}
	Since $K\in\mathcal{H}_\infty$, for some $r_0\in(1,q_0)$, we know that $K\in\mathcal{H}_{r_{0}'}$, hence, we can apply $(\ref{eq 3.4})$-$(\ref{eq 3.6})$ to get that
	$$\|\mathcal{M}_{\mathcal{V}_\rho(\mathcal{T}),Q_0}\|_{L^{q_0}\rightarrow L^{q_0}}\leq C(n,q_0,\|\mathcal{V}_\rho(\mathcal{T})\|_{L^{q_0}\rightarrow L^{q_0}}).$$
	This together with $(P_1)$, we deduce that
	$$|\{x\in\mathbb{R}^n:\mathcal{M}_{\mathcal{V}_\rho(\mathcal{T}),Q_0}g(x)>\lambda/2\}|\leq
	C(n,q_0,\|\mathcal{V}_\rho(\mathcal{T})\|_{L^{q_0}\rightarrow L^{q_0}})\frac{1}{\lambda}\|f\|_{L^1(\mathbb{R}^n)}.$$
	Denote $\widetilde{Q}:=\bigcup_j25\sqrt{n}Q_j$, by $(P_4)$, we see that
	\begin{align*}
		&|\{x\in\mathbb{R}^n:\mathcal{M}_{\mathcal{V}_\rho(\mathcal{T}),Q_0}b(x)>\lambda/2\}|\\
		&\quad\leq|\widetilde{Q}|+|\{x\in\mathbb{R}^n\backslash\widetilde{Q}:\mathcal{M}_{\mathcal{V}_\rho
			(\mathcal{T}),Q_0}b(x)>\lambda/2\}|\\
		&\quad\leq C_n\frac{1}{\lambda}\|f\|_{L^1(\mathbb{R}^n)}+|\{x\in\mathbb{R}^n\backslash\widetilde{Q}:
		\mathcal{M}_{\mathcal{V}_\rho(\mathcal{T}),Q_0}b(x)>\lambda/2\}|.
	\end{align*}
	Note that for any $x\in \mathbb{R}^n$,
	$$\mathcal{M}_{\mathcal{V}_\rho(\mathcal{T}),Q_0}b(x)\leq CMb(x)+\mathcal{\widetilde{M}}_{\mathcal{V}_\rho(\mathcal{T}),Q_0}b(x),$$
	and
	\begin{equation*}
		\mathcal{\widetilde{M}}_{\mathcal{V}_\rho(\mathcal{T}),Q_0}f(x):=
		\begin{cases}
			\sup_{Q\ni x,Q\subset Q_0}\esss_{\xi\in Q}|\mathcal{V}_{\rho}(\mathcal{T}f\chi_{3Q_0\backslash B_x})(\xi)|, ~x\in Q_0;\\
			0,~$otherwise$.\\
		\end{cases}
	\end{equation*}
	Hence, we only need to show the weak type (1,1) of $\mathcal{\widetilde{M}}_{\mathcal{V}_\rho(\mathcal{T}),Q_0}$. Let $I_i:=(\varepsilon_i,\varepsilon_{i+1}]$ and $A_{I_i}(\xi):=\{y\in\mathbb{R}^n:|\xi-y|\in I_i\}$. For $x\in Q_0\backslash \widetilde{Q}$, choose $Q\ni x, \xi\in Q$ and $\{\varepsilon_i\}_i$ such that
	$$\mathcal{\widetilde{M}}_{\mathcal{V}_\rho(\mathcal{T}),Q_0}b(x)\leq2\Big[\sum_i\Big|\sum_j
	T(\chi_{A_{I_i}(\xi)}\chi_{3Q_0\backslash B_x}b_j)(\xi)\Big|^\rho\Big]^{1/\rho}.$$
	Let us consider the following three sets of indices $j's$:
	$$L_{I_i}^1(\xi):=\{j:Q_j\subset A_{I_i}(\xi)\cap(3Q_0\backslash B_x)\},$$
	$$L_{I_i}^2(\xi):=\{j:Q_j\nsubseteq A_{I_i}(\xi)\cap(3Q_0\backslash B_x),Q_j\cap A_{I_i}(\xi)\cap(3Q_0\backslash B_x)\neq\emptyset,Q_j\cap\partial(3Q_0)\neq\emptyset\},$$
	\begin{align*}
		L_{I_i}^3(\xi)&:=\{j:Q_j\nsubseteq A_{I_i}(\xi)\cap(3Q_0\backslash B_x),Q_j\cap A_{I_i}(\xi)\cap(3Q_0\backslash B_x)\neq\emptyset,\\
		&\qquad\quad Q_j\cap\partial(B_x)\neq\emptyset~or~Q_j\cap\partial(A_{I_i}(\xi)\neq\emptyset\}.
	\end{align*}
	Then we get
	\begin{align*}
		\Big[\sum_i\Big|\sum_j
		T(\chi_{A_{I_i}(\xi)}\chi_{3Q_0\backslash B_x}b_j)(\xi)\Big|^\rho\Big]^{1/\rho}&\leq\sum_{m=1}^{3}\Big[\sum_i\Big|\sum_{j\in L_{I_i}^{m}(\xi)}
		T(\chi_{A_{I_i}(\xi)}\chi_{3Q_0\backslash B_x}b_j)(\xi)\Big|^\rho\Big]^{1/\rho}\\
		&=:\sum_{m=1}^{3}L_m(\xi)(x).
	\end{align*}
	By $(P_3)$, it implies that
	\begin{align*}
		L_1(\xi)(x)&\leq\sum_i\sum_{j\in L_{I_i}^{1}(\xi)}\int_{\mathbb{R}^n}|K(\xi,y)-K(\xi,y_j)|\chi_{A_{I_i}(\xi)}(y)\chi_{3Q_0
			\backslash B_x}(y)|b_j(y)|dy\\
		&\leq\sum_j\int_{\mathbb{R}^n}|K(\xi,y)-K(\xi,y_j)|\chi_{3Q_0
			\backslash B_x}(y)|b_j(y)|dy,
	\end{align*}
	where $y_j$ is the center of $Q_j$. Note that for $x\in Q_0\backslash\widetilde{Q}$ and $y\in Q_j\bigcap(3Q_0\backslash B_x)$,
	\begin{align*}
		|x-y|\geq|x-y_j|-\sqrt{n}l_{Q_j}\geq\frac{23}{25}|x-y_j|,
	\end{align*}
	and
	\begin{align*}
		|\xi-y|\geq|x-y|-|x-\xi|\geq|x-y|-\frac{|x-y|}{9\sqrt{n}}\geq\frac{8}{9}|x-y|.
	\end{align*}
	Then
	\begin{align*}
		|\xi-y_j|&\geq|\xi-y|-|y-y_j|\\
		&\geq\frac{8}{9}|x-y|-\sqrt{n}l_{Q_j}\geq\frac{8}{9}\cdot
		\frac{23}{25}|x-y_j|-\frac{2}{25}|x-y_j|\geq7\sqrt{n}l_{Q_j}.
	\end{align*}
	Since $K\in \mathcal{H}_\infty$, by $(P_3)$ and $(P_4)$, we have
	\begin{align*}
		&|\{x\in Q_0\backslash\widetilde{Q}:L_1(\xi)(x)>\frac{\lambda}{16}\}|\\
		&\quad\leq\frac{16}{\lambda}\sum_j\int_{Q_0\backslash\hat{B}_j}\int_{\mathbb{R}^n}
		|K(\xi,y)-K(\xi,y_j)|\chi_{3Q_0\backslash B_x}(y)|b_j(y)|dydx\\
		&\quad\leq C\frac{1}{\lambda}\sum_j\int_{\mathbb{R}^n}|b_j(y)|\sum_{k=1}^{\infty}\int_{2^k\hat{B}_j
			\backslash2^{k-1}\hat{B}_j}|K(\xi,y)-K(\xi,y_j)|\chi_{3Q_0\backslash B_x}(y)dxdy\\
		&\quad\leq C_n\frac{1}{\lambda}\sum_j\int_{\mathbb{R}^n}|b_j(y)|\sup_{\hat{B}_j}\sup_{y,y_j\in\frac{1}{2}
			\hat{B}_j}\sum_{k=1}^{\infty}|2^{k}\hat{B}_j|\esss_{\xi\in2^{k}\hat{B}_j\backslash2^{k-1}
			\hat{B}_j}|K(\xi,y)-K(\xi,y_j)|dy\\
		&\quad\leq C_n\frac{1}{\lambda}\|f\|_{L^1(\mathbb{R}^n)},
	\end{align*}
	where $\hat{B}_j:=B(y_j,7\sqrt{n}l_{Q_j})$. By the same arguments as in \cite{DuLY}, we can obtain
	$$|\{x\in Q_0\backslash\widetilde{Q}:L_m(\xi)(x)>\frac{\lambda}{16}\}|\leq C_n\frac{1}{\lambda}\|f\|_{L^1(\mathbb{R}^n)}, \quad m=2,3.$$
	This implies the desired conclusion and completes the proof of Lemma $\ref{lm 3.2}$.
\end{proof}

Now we are in a position to prove Theorem $\ref{Theorem1.1}$.
\begin{proof}[Proof of Theorem $\ref{Theorem1.1}$]
	The proof of this theorem follows the standard step in \cite{LOR1}. We only prove $(\ref{(1.6)})$ since $(\ref{(1.5)})$ follows the similar pattern.
	In view of Remark \ref{Remark 2.4}, we see that there exist $3^n$ dyadic lattices $\mathcal{D}_j$ such that for every cube $Q\subset\mathbb{R}^n$, there is a cube $R_Q\in \mathcal{D}_j$ for some $j$, for which $3Q\subset R_Q$ and $|R_Q|\leq 9^n|Q|$.
	
	Fixed a cube $Q_0\subset \mathbb{R}^n$, we first claim that there exists a $\frac{1}{2}$-sparse family $\mathcal{F}\subseteq \mathcal{D}(Q_0)$ such that for a.e. $x\in Q_0$,
	\begin{equation}\label{eq 3.7}
		\begin{split}
			\mathcal{V}_\rho(\mathcal{T}_bf\chi_{3Q_0})(x) &\leq C(n,r',q_0,\|\mathcal{V}_\rho(\mathcal{T})\|_{L^{q_0}\rightarrow L^{q_0}})\\
			&\quad\times\sum_{Q\in\mathcal{F}}(|b(x)-b_{R_Q}|\langle|f|^{r'}\rangle_{3Q}
			^{1/r'}+\langle|f(b-b_{R_Q})|^{r'}\rangle_{3Q}^{1/r'})\chi_{Q}(x).
		\end{split}
	\end{equation}
	For some $\alpha_n>0$, which will be determined later, denote $E=E_1\cup E_2$, where
	\begin{align*}
		E_1:&=\{x\in Q_0: |f(x)|>\alpha_n\langle|f|^{r'}\rangle_{3Q_0}^{1/r'}\}\\
		&\quad\cup\{x\in Q_0: \mathcal{M}_{\mathcal{V}_{\rho}(\mathcal{T}),Q_0}(f)(x)>\alpha_n
		\|\mathcal{M}_{\mathcal{V}_{\rho}(\mathcal{T}),Q_0}\|_{L^{r',\infty}\rightarrow L^{r'}})
		\langle|f|^{r'}\rangle_{3Q_0}^{1/r'}\},
	\end{align*}
	and
	\begin{align*}
		E_2:&=\{x\in Q_0: |f(x)(b(x)-b_{R_{Q_0}})|>\alpha_n\langle|f(b-b_{R_{Q_0}})|^{r'}\rangle_{3Q_0}^{1/r'}\}\\
		&\quad\cup\{x\in Q_0: \mathcal{M}_{\mathcal{V}_{\rho}(\mathcal{T}),Q_0}(f(b-b_{R_{Q_0}}))(x)>\alpha_n
		\|\mathcal{M}_{\mathcal{V}_{\rho}(\mathcal{T}),Q_0}\|_{L^{r',\infty}\rightarrow L^{r'}})\\
		&\quad\quad\times\langle|f(b-b_{R_{Q_0}})|^{r'}\rangle_{3Q_0}^{1/r'}\}.
	\end{align*}
	By Lemma $\ref{lm 3.2}$,
	\begin{equation*}
		\begin{split}
			|E|&\leq\frac{\int_{Q_0}|f(x)|dx}{\alpha_n\langle|f|^{r'}\rangle_{3Q_0}^{1/r'}}+
			\frac{\int_{Q_0}|f(x)(b(x)-b_{R_{Q_0}})|dx}{\alpha_n\langle|f(b-b_{R_{Q_0}})|^{r'}
				\rangle_{3Q_0}^{1/r'}}+\int_{3Q_0}\frac{|f(x)(b(x)-b_{R_{Q_0}})|^{r'}}{\alpha_{n}^{r'}
				\langle|f(b-b_{R_{Q_0}})|^{r'}\rangle_{3Q_0}}dx\\
			&\quad+\int_{3Q_0}\frac{|f(x)|^{r'}}{\alpha_{n}^{r'}
				\langle|f|^{r'}\rangle_{3Q_0}}dx\\
			&\leq(\frac{2\cdot3^n}{\alpha_n}+\frac{2\cdot3^n}{\alpha_n^{r'}})|Q_0|.
		\end{split}
	\end{equation*}
	Hence, taking $\alpha_n$ large enough, we deduce that
	$$|E|\leq\frac{1}{2^{n+1}}|Q_0|.$$
	By using the Calder\'{o}n-Zygmund decomposition to the function $\chi_E$ on $Q_0$ at height $\lambda=\frac{1}{2^{n+1}}$, we obtain pairwise disjoint cubes $P_j$ such that
	$$\chi_E(x)\leq \frac{1}{2^{n+1}}$$
	for a.e. $x\not\in\cup P_j$. Together with this we immediately obtain $\Big|E\backslash\bigcup_j P_j\Big|=0$, we also have
	$$\frac{1}{2^{n+1}}|P_j|\leq|P_j\cap E|\leq\frac{1}{2}|P_j|,$$
	and that $\sum_{j}|P_j|\leq1/2|Q_0|,~~P_j\cap E^{c}\neq\emptyset$.
	As the direct results of Lemma $\ref{lm 3.2}$, we have
	\begin{equation}\label{eq 3.8}
		\esss_{\xi\in P_j}\mathcal{V}_{\rho}(\mathcal{T}f\chi_{3Q_0\backslash3P_j})(\xi)\leq
		C(n,r',q_0,\|\mathcal{V}_\rho(\mathcal{T})\|_{L^{q_0}\rightarrow L^{q_0}})\langle|f|^{r'}\rangle_{3Q_0}^{1/r'},
	\end{equation}
	and
	\begin{equation}\label{eq 3.9}\begin{array}{ll}
			&\esss_{\xi\in P_j}\mathcal{V}_{\rho}(\mathcal{T}(b-b_{R_{Q_0}})f\chi_{3Q_0\backslash3P_j})(\xi)\\
			&\qquad\qquad\qquad\leq
			C(n,r',q_0,\|\mathcal{V}_\rho(\mathcal{T})\|_{L^{q_0}\rightarrow L^{q_0}})\langle|f(b-b_{R_{Q_0}})|^{r'}\rangle_{3Q_0}^{1/r'}.\end{array}
	\end{equation}
	By Lemmas $\ref{lm 3.1}$ and $\ref{lm 3.2}$, for almost everywhere $x\in Q_0\backslash\cup P_j$,
	\begin{equation}\label{eq 3.10}
		\mathcal{V}_{\rho}(\mathcal{T}f\chi_{3Q_0})(x)\leq C(n,r',q_0,\|\mathcal{V}_\rho(\mathcal{T})\|_{L^{q_0}\rightarrow L^{q_0}})
		\langle|f|^{r'}\rangle_{3Q_0}^{1/r'},
	\end{equation}
	and
	\begin{equation}\label{eq 3.11}
		\begin{split}
			\mathcal{V}_{\rho}(\mathcal{T}(b-b_{R_{Q_0}})f\chi_{3Q_0})(x)&\leq
			C(n,r',q_0,\|\mathcal{V}_\rho(\mathcal{T})\|_{L^{q_0}\rightarrow L^{q_0}})\langle|f(b-b_{R_{Q_0}})|^{r'}\rangle_{3Q_0}^{1/r'}.
		\end{split}
	\end{equation}
	Note that for any pairwise disjoint cubes $P_j$ we have obtained,
	\begin{equation*}
		\begin{split}
			\mathcal{V}_{\rho}(\mathcal{T}_bf\chi_{3Q_0})(x)\chi_{Q_0}(x)&=\mathcal{V}_{\rho}(\mathcal{T}_b
			f\chi_{3Q_0})(x)\chi_{Q_0\backslash\bigcup_{j}P_j}(x)+\sum_{j}\mathcal{V}_{\rho}
			(\mathcal{T}_bf\chi_{3Q_0})(x)\chi_{P_j}(x)\\
			&\leq\mathcal{V}_{\rho}(\mathcal{T}_bf\chi_{3Q_0})(x)\chi_{Q_0\backslash\bigcup_{j}P_j}(x)
			+\sum_{j}\mathcal{V}_{\rho}(\mathcal{T}_bf\chi_{3P_j})(x)\chi_{P_j}(x)\\
			&\quad+\sum_{j}\mathcal{V}_{\rho}(\mathcal{T}_bf\chi_{3Q_0\backslash3P_j})(x)\chi_{P_j}(x).
		\end{split}
	\end{equation*}
	For $c=b-b_{R_{Q_0}}$, using the fact that $\mathcal{V}_{\rho}(\mathcal{T}_bf)=\mathcal{V}_{\rho}(\mathcal{T}_{b-c}f)$, we have
	\begin{equation*}
		\begin{split}
			&\mathcal{V}_{\rho}(\mathcal{T}_bf\chi_{3Q_0})(x)\chi_{Q_0\backslash\bigcup_{j}P_j}(x)+
			\sum_{j}\mathcal{V}_{\rho}(\mathcal{T}_bf\chi_{3Q_0\backslash3P_j})(x)\chi_{P_j}(x)\\
			&\quad\leq|b(x)-b_{R_{Q_0}}|\Big(\mathcal{V}_{\rho}(\mathcal{T}f\chi_{3Q_0})(x)\chi_{Q_0\backslash
				\bigcup_{j}P_j}(x)+\sum_{j}\mathcal{V}_{\rho}(\mathcal{T}f\chi_{3Q_0\backslash3P_j})(x)
			\chi_{P_j}(x)\Big)\\
			&\quad\quad+\mathcal{V}_{\rho}(\mathcal{T}(b-b_{R_{Q_0}})f\chi_{3Q_0})(x)\chi_{Q_0\backslash
				\bigcup_{j}P_j}(x)\\
			&\quad\quad+\sum_{j}\mathcal{V}_{\rho}(\mathcal{T}(b-b_{R_{Q_0}})f
			\chi_{3Q_0\backslash3P_j})(x)\chi_{P_j}(x).
		\end{split}
	\end{equation*}
	By $(\ref{eq 3.8})$-$(\ref{eq 3.11})$, for a.e. $x\in Q_0$, we get
	\begin{equation*}
		\begin{split}
			\mathcal{V}_{\rho}(\mathcal{T}_bf\chi_{3Q_0})(x)\chi_{Q_0}(x)&\leq C(n,r',q_0,\|\mathcal{V}_\rho(\mathcal{T})\|_{L^{q_0}\rightarrow L^{q_0}}) \Big(|b(x)-b_{R_{Q_0}}|\langle|f|^{r'}\rangle_{3Q_0}^{1/r'}\\
			&\quad+\langle|(b-b_{R_{Q_0}})f|^{r'}\rangle_{3Q_0}^{1/r'}\Big)
			+\sum_{j}\mathcal{V}_{\rho}(\mathcal{T}_bf\chi_{3P_j})(x)\chi_{P_j}(x).
		\end{split}
	\end{equation*}
	Iterating the above inequality, we obtain a $\frac{1}{2}$-sparse family $\mathcal{F}=\{P_{j}^{k}\}(k\in \mathbb{Z}^+)$ with $\sum_{j}|P_j|\leq\frac{1}{2}|Q_0|$, where $\{P_{j}^{0}\}=\{Q_0\}$, $\{P_{j}^{1}\}=\{P_j\}$ and $\{P_{j}^{k}\}$ are the cubes achieved at the $k$-stage of the iterative process. This proves $(\ref{eq 3.7})$.
	
	In the end of the proof, let us show that $\text {supp}~ f\subset 3Q_j$ with $\bigcup_{j}Q_j=\mathbb{R}^n$. We begin by taking a cube $Q_0$ such that $\text {supp}~ f\subset Q_0$. And cover $3Q_0\backslash Q_0$ by $3^n-1$ congruent cubes $Q_j$. For every $j$, $Q_0\subset 3Q_j$. We continue to do the same way for $9Q_0\backslash 3Q_0$ and so on. It's easy to check that the union of the cubes $Q_j$ of this process, including $Q_0$, satisfies our requirement. Applying to each $Q_j$, then for a.e. $x\in Q_j$,
	\begin{align*}
		\mathcal{V}_{\rho}(\mathcal{T}_bf)(x)\chi_{Q_j}(x)&\leq
		C(n,r',q_0,\|\mathcal{V}_\rho(\mathcal{T})\|_{L^{q_0}\rightarrow L^{q_0}}) \sum_{Q\in\mathcal{F}_j}\Big(|b(x)-b_{R_Q}|\langle|f|^{r'}\rangle_{3Q}^{1/r'}\\
		&\quad+\langle|f(b-b_{R_Q})|^{r'}\rangle_{3Q}^{1/r'}\Big)\chi_Q(x),
	\end{align*}
	where $\mathcal{F}_j\subset \mathcal{D}(Q_j)$ is a $\frac{1}{2}$-sparse family. Let $\mathcal{F}=\bigcup_{j}\mathcal{F}_j$, we also have that $\mathcal{F}$ is a $\frac{1}{2}$-sparse family and the following estimate holds for a.e. $x\in \mathbb{R}^n$,
	\begin{align*}
		\mathcal{V}_{\rho}(\mathcal{T}_bf)(x)&\leq
		C(n,r',q_0,\|\mathcal{V}_\rho(\mathcal{T})\|_{L^{q_0}\rightarrow L^{q_0}}) \sum_{Q\in\mathcal{F}}\Big(|b(x)-b_{R_Q}|\langle|f|^{r'}\rangle_{3Q}^{1/r'}\\
		&\quad+\langle|f(b-b_{R_Q})|\rangle_{3Q}^{1/r'}\Big)\chi_Q(x).
	\end{align*}
	Denote $S_j=\{R_Q\in \mathcal{D}_j:Q\in \mathcal{F}\}$ and note that $|f|_{3Q}\leq C_n|f|_{3R_Q}$, it yields that
	\begin{align*}
		\mathcal{V}_{\rho}(\mathcal{T}_bf)(x)&\leq C(n,r',q_0,\|\mathcal{V}_\rho(\mathcal{T})\|_{L^{q_0}\rightarrow L^{q_0}})\sum_{j=1}^{3^n}\sum_{R\in S_j}\Big(|b(x)-b_{R}|\langle|f|^{r'}\rangle_{R}^{1/r'}\\
		&\quad+\langle|f(b-b_{R})|^{r'}\rangle_{R}^{1/r'}\Big)\chi_R(x).
	\end{align*}
	This completes the proof of Theorem $\ref{Theorem1.1}$.
\end{proof}

\begin{proof}[Proof of Theorem $\ref{Theorem1.2}$]
	The proof directly follows from $(\ref{(1.5)})$ and $(\ref{eq 2.1})$-$(\ref{eq 2.2})$.
\end{proof}

\section{Proofs of the Theorems \ref{Theorem1.4} and \ref{Theorem1.5}}\label{section 4}
This section is devoted to the proofs of Theorems $\ref{Theorem1.4}$ and $\ref{Theorem1.5}$. To prove $\ref{Theorem1.4}$, we need the following lemma.
\begin{lemma}\label{new}\,\
	If $\omega\in A_{p/r'}$ with $1\leq r'<\infty$, then there exists a constant $s>1$ with $s'\leq c_{n,p}[\omega]_{A_\infty}$ such that $[\omega]_{A_{p/(r's)}}\lesssim[\omega]_{A_{p/r'}}$.
\end{lemma}
\begin{proof}
	Take $\theta=\frac{1}{\tau_n[\omega]_{A_\infty}}$, by Lemma $\ref{lm 2.1}$, we have
	\begin{align*}
		\Big(\frac{1}{|Q|}\int_Q\omega(x)^{(1-(p/r')')(1+\theta)}dx\Big)^{\frac{p/r'-1}{1+\theta}}\leq
		2^{p-1}\Big(\frac{1}{|Q|}\int_Q\omega(x)^{1-(p/r')'}dx\Big)^{p/r'-1}.
	\end{align*}
	From this, we obtain
	\begin{align*}
		&\Big(\frac{1}{|Q|}\int_Q\omega(x)dx\Big)\Big(\frac{1}{|Q|}\int_Q\omega(x)^{(1-(p/r')')(1+\theta)}dx
		\Big)^{\frac{p/r'-1}{1+\theta}}\\
		&\quad\leq2^{p-1}\Big(\frac{1}{|Q|}\int_Q\omega(x)dx\Big)
		2^{p-1}\Big(\frac{1}{|Q|}\int_Q\omega(x)^{1-(p/r')'}dx\Big)^{p/r'-1}.
	\end{align*}
	Choosing $p/(r's)-1=\frac{p/r'-1}{1+\theta}$, then $s'\leq c_{n,p}[\omega]_{A_\infty}$ and
	$[\omega]_{A_{p/(r's)}}\leq2^{p-1}[\omega]_{A_{p/r'}}$.
\end{proof}

\begin{proof}[Proof of Theorem $\ref{Theorem1.4}$]
	Instead of using the conjugation method, which relies on the sparse bounds $(\ref{(1.5)})$. We use the sparse bounds $(\ref{(1.6)})$. Let us denote
	$$\mathcal{A}_{\mathcal{S}_j,b}^{1}f(x)=\sum_{R\in \mathcal{S}_j}|b(x)-b_{R}|\langle|f|^{r'}\rangle_{R}^{1/r'}\chi_{R}(x),\quad \mathcal{A}_{\mathcal{S}_j,b}^{2}f(x)=\sum_{R\in \mathcal{S}_j}\langle|f(b-b_{Q})|^{r'}\rangle_{R}^{1/r'}\chi_R(x).$$
	First, we consider $\mathcal{A}_{\mathcal{S}_j,b}^{1}$, using $\|b-b_R\|_{\exp L(\omega),Q}\leq c_n[\omega]_{A_\infty}\|b\|_{\rm BMO}$ (see \cite{IR}) and the generalized H\"{o}lder inequality $(\ref{eq 2.3})$,
	\begin{equation*}
		\begin{split}
			\|\mathcal{A}_{\mathcal{S}_j,b}^{1}f\|_{L^{p}(\omega)}&=\sup_{\|g\|_{L^{p'}(\omega)}\leq1}
			\int_{\mathbb{R}^n}\mathcal{A}_{\mathcal{S}_j,b}^{1}f(x)g(x)\omega(x)dx\\
			&=\sup_{\|g\|_{L^{p'}(\omega)}\leq1}\sum_{R\in \mathcal{S}_j}\frac{1}{\omega(R)}\int_{R}|b(x)-b_R|g(x)\omega(x)dx\omega(R)
			\langle|f|^{r'}\rangle_{R}^{1/r'}\\
			&\leq\sup_{\|g\|_{L^{p'}(\omega)}\leq1}\sum_{R\in \mathcal{S}_j}\|b-b_R\|_{\exp L(\omega),R}\|g\|_{L(\log L)(\omega),R}\omega(R)
			\langle|f|^{r'}\rangle_{R}^{1/r'}\\
			&\leq c_n[\omega]_{A_\infty}\|b\|_{\rm BMO}\sup_{\|g\|_{L^{p'}(\omega)}\leq1}\sum_{R\in \mathcal{S}_j}\|g\|_{L(\log L)(\omega),R}\omega(R)\langle|f|^{r'}\rangle_{R}^{1/r'}.
		\end{split}
	\end{equation*}
	Let $\mathcal{B}$ be the family of the principal cubes in the usual sense, i.e., $\mathcal{B}=\bigcup_{k=0}^{\infty}\mathcal{B}_k$, where $\mathcal{B}_0:=\{\rm maximal~ cubes ~ in ~ \mathcal{S}_j\}$ and
	$$\mathcal{B}_{k+1}:=\bigcup_{B\in\mathcal{B}_k}ch_{\mathcal{B}}(B),\quad where\quad ch_{\mathcal{B}}(B)=\{R\subseteq\mathcal{B}~\rm maximal~s.t.~\tau(R)>2\tau(B)\},$$
	where $\tau(R)=\|g\|_{L(\log L)(\omega),R}\langle|f|^{r'}\rangle_{R}^{1/r'}$. Then
	\begin{equation*}
		\begin{split}
			\sum_{R\in S_j}\|g\|_{L(\log L)(\omega),R}\omega(R)\langle|f|^{r'}\rangle_{R}^{1/r'}
			&\leq\sum_{B\in \mathcal{B}}\|g\|_{L(\log L)(\omega),B}\langle|f|^{r'}\rangle_{B}^{1/r'}\sum_{R\in S_j,\pi(R)=B}\omega(R)\\
			&\leq c_n[\omega]_{A_\infty}\sum_{B\in \mathcal{B}}\|g\|_{L(\log L)(\omega),B}\langle|f|^{r'}\rangle_{B}^{1/r'}\omega(B)\\
			&\leq c_n[\omega]_{A_\infty}\int_{\mathbb{R}^n}M_{r'}f(x)M_{L\log L(\omega)}g(x)\omega(x)dx\\
			&\leq c_n[\omega]_{A_\infty}\int_{\mathbb{R}^n}M_{r'}f(x)M_{\omega}^{2}g(x)\omega(x)dx,
		\end{split}
	\end{equation*}
	where $\pi(R)$ is the minimal principal cube which contains $R$. Hence, using the sharp bounds of $M_{r'}$ (see \cite{HP}) and $\|M_\omega g\|_{L^{p'}(\omega)}\leq c_{n,p}\|g\|_{L^{p'}(\omega)}$ (see \cite{HP}), we obtain
	\begin{equation*}
		\begin{split}
			\|\mathcal{A}_{\mathcal{S}_j,b}^{1}f\|_{L^{p}(\omega)}&\leq c_n[\omega]_{A_\infty}^{2}\|b\|_{\rm BMO}\sup_{\|g\|_{L^{p'}(\omega)}\leq1}\int_{\mathbb{R}^n}M_{r'} f(x)M_{\omega}^{2}g(x)\omega(x)dx\\
			&\leq c_n[\omega]_{A_\infty}^{2}\|b\|_{\rm BMO}\|M_{r'} f\|_{L^p(\omega)}\sup_{\|g\|_{L^{p'}(\omega)}\leq1}\|M_{\omega}^{2}g\|_{L^{p'}(\omega)}\\
			&\leq c_{n,p} [\omega]_{A_\infty}^{2}([\omega]_{A_{p/r'}}[\omega^{-\frac{r'}{p-r'}}]_{A_\infty})
			^{1/p}\|b\|_{\rm BMO}\|f\|_{L^p(\omega)}.
		\end{split}
	\end{equation*}
	
	Now, turn to $\mathcal{A}_{\mathcal{S}_j,b}^{2}f(x)$. Applying the John-Nirenberg theorem, we deduce that
	\begin{align*}
		&\Big(\frac{1}{|R|}\int_R|b(x)-\langle b\rangle_R|^{r's'}\Big)^{1/(r's')}\\
		&\quad=\Big(r's'\int_{0}^{\infty}t^{r's'-1}|R|^{-1}|\{x\in R:|b(x)-\langle b\rangle_R|>t\}|dt\Big)^{1/(r's')}\\
		&\quad\leq e^2\Big(\int_{0}^{\infty}t^{r's'-1}e^{\frac{-t}{2^ne\|b\|_{\rm BMO}}}dt\Big)^{1/(r's')},
	\end{align*}
	where the last inequality is due to the fact that $(r's')^{1/(r's')}\leq e$. Since $$t^{r's'-1}e^{\frac{-t}{2^{n+1}e\|b\|_{\rm BMO}}}\leq(2^{n+1}e)^{r's'-1}(r's'-1)^{r's'-1}\|b\|_{\rm BMO}^{r's'-1},$$ we have
	\begin{align}\label{new inequality}\,\
		&\Big(\frac{1}{|R|}\int_R|b(x)-\langle b\rangle_R|^{r's'}\Big)^{1/(r's')}\nonumber\\
		&\quad\leq e^2\Big(\int_{0}^{\infty}(2^{n+1}e)^{r's'-1}(r's'-1)^{r's'-1}\|b\|_{\rm BMO}^{r's'-1}e^{\frac{-t}{2^{n+1}e\|b\|_{\rm BMO}}}dt\Big)^{1/(r's')}\\
		&\quad=2^{n+1}e^3r's'\|b\|_{\rm BMO}\nonumber.
	\end{align}
	As we do in dealing with $A_{\mathcal{S}_j,b}^{1}f$, by $(\ref{new inequality})$ and take $s'$ as Lemma $\ref{new}$, we get that
	\begin{align*}
		\|A_{\mathcal{S}_j,b}^{2}f\|_{L^{p}(\omega)}&=\sup_{\|g\|_{L^{p'}(\omega)}\leq1}
		\int_{\mathbb{R}^n}\mathcal{A}_{\mathcal{S}_j,b}^{2}f(x)g(x)\omega(x)dx\\
		&=\sup_{\|g\|_{L^{p'}(\omega)}\leq1}\sum_{R\in \mathcal{S}_j}\frac{1}{\omega(R)}\int_{R}g(x)\omega(x)dx \omega(R)
		\langle|f(b-\langle b\rangle_{R})|^{r'}\rangle_{R}^{\frac{1}{r'}}\\
		&\leq\sup_{\|g\|_{L^{p'}(\omega)}\leq1}\sum_{R\in \mathcal{S}_j}\frac{1}{\omega(R)}\int_{R}g(x)\omega(x)dx \omega(R)\langle|b-\langle b\rangle_R|^{r's'}\rangle_R^{1/(r's')}\|f\|_{L^{r's},R}\\
		&\leq c_{n,p,r'}\|b\|_{\rm BMO}[\omega]_{A_\infty}^2\sup_{\|g\|_{L^{p'}(\omega)}\leq1} \int_{\mathbb{R}^n}M_\omega g(x)M_{r's}f(x)\omega(x)dx.
	\end{align*}
	Now, using the following results given in \cite{Hy2,HP}:
	\begin{align*}
		\|Mf\|_{L^p(\omega)}\leq4ep'([\omega]_{A_p}[\omega^{-1/(p-1)}]_{A_\infty})^{1/p}\|f\|_{L^p(\omega)}
		,\quad \omega\in A_p,
	\end{align*}
	\begin{align*}
		[\omega^{-\frac{1}{p-1}}]_{A_\infty}\leq[\omega]_{A_p}^{\frac{1}{p-1}},\quad \omega\in A_p.
	\end{align*}
	Then by Lemma $\ref{new}$ and $(p/(r's))'\leq\frac{p+r'}{p-r'}$,
	\begin{align*}
		\|M_{r's}f\|_{L^p(\omega)}&=\|M(|f|^{r's})\|_{L^{p/(r's)}(\omega)}^{1/(r's)}\\
		&\leq
		[4e(p/(r's))'([\omega]_{A_{p/(r's)}}[\omega^{-\frac{1}{p/(r's)-1}}]_{A_\infty})^{r's/p}]^{1/(r's)}
		\||f|^{r's}\|_{L^{p/(r's)}}^{1/(r's)}\\
		&\lesssim[(p/(r's))']^{1/(r's)}([\omega]_{A_{p/r'}}[\omega]_{A_{p/r'}}^{(p/(r's))'-1})^{1/p}
		\|f\|_{L^p(\omega)}\\
		&\lesssim[\omega]_{A_{p/r'}}^{(p/(r's))'\frac{1}{p}}\|f\|_{L^p(\omega)}\leq[\omega]_{A_{p/r'}}
		^{\frac{p+r'}{p(p-r')}}\|f\|_{L^p(\omega)}.
	\end{align*}
	Therefore,
	\begin{equation*}
		\begin{split}
			\|\mathcal{A}_{\mathcal{S}_j,b}^{2}f\|_{L^{p}(\omega)}&\leq c_{n,p,r'}[\omega]_{A_\infty}^2\|b\|_{\rm BMO}\sup_{\|g\|_{L^{p'}(\omega)}\leq1}\int_{\mathbb{R}^n}M_{r's} f(x)M_{\omega}g(x)\omega(x)dx\\
			&\leq c_{n,p,r'}[\omega]_{A_\infty}^2\|b\|_{\rm BMO}\|M_{r's} f\|_{L^p(\omega)}\sup_{\|g\|_{L^{p'}(\omega)}\leq1}\|M_{\omega}g\|_{L^{p'}(\omega)}\\
			&\leq c_{n,p,r'} [\omega]_{A_\infty}^2[\omega]_{A_{p/r'}}^{\frac{p+r'}{p(p-r')}}
			\|b\|_{\rm BMO}\|f\|_{L^p(\omega)}.
		\end{split}
	\end{equation*}
	This, together with the estimate of $\|\mathcal{A}_{\mathcal{S}_j,b}^{1}f\|_{L^{p}(\omega)}$, deduces that
	\begin{align*}
		\|\mathcal{V}_{\rho}(\mathcal{T}_bf)\|_{L^p(\omega)}&\leq C(n,r',p,q_0,\|\mathcal{V}_\rho(\mathcal{T})\|_{L^{q_0}\rightarrow L^{q_0}}) [\omega]_{A_\infty}^{2}\\
		&\quad\times\Big([\omega]_{A_{p/r'}}^{\frac{p+r'}{p(p-r')}}+([\omega]_{A_{p/r'}}
		[\omega^{-\frac{r'}{p-r'}}]_{A_\infty})^{1/p}\Big)\|b\|_{\rm BMO}\|f\|_{L^p(\omega)}.
	\end{align*}
	Theorem $\ref{Theorem1.4}$ is proved.
\end{proof}

\begin{proof}[Proof of Theorem $\ref{Theorem1.5}$]
	The case $r'=1$ has been obtained in \cite{IR}, here, we extend the results to $1\leq r'<\infty$.
	We will apply Lemma $\ref{lm 2.6}$ to prove our theorem. Choose $\psi_j(t)=t^{r'}\log(e+t)^{r'j}$, $j=0,1$. It's not hard to see that $\psi_j$ satisfies the conditions of Lemma $\ref{lm 2.6}$. For $j=0$,
	\begin{equation*}
		\begin{split}
			&\int_{1}^{\infty}\frac{\varphi_{0}^{-1}\circ\Phi_{1}^{-1}(t)(\log(e+t))^{4r'}}
			{t^2(\log(e+t))^4}dt\\
			&\quad=\int_{\Phi_{1}^{-1}(1)}^{\infty}\frac{\varphi_{0}^{-1}(t)\Phi_{1}^{'}(t)(\log(e+\Phi_1(t)))
				^{4r'}}{\Phi_{1}^{2}(t)(\log(e+\Phi_1(t)))^{4}}dt\\
			&\quad\leq  \int_{\Phi_{1}^{-1}(1)}^{\infty}\frac{\varphi_{0}^{-1}(t)[\log(e+\Phi_{1}(t))]^{4r'-4}}{t^{2}
				\log(e+t)}dt.
		\end{split}
	\end{equation*}
	Take $\varphi_0(t)=t[\log(e+t)]^{4r'-4}[\log(e+\log(e+t))]^{1+\varepsilon}$. Then
	\begin{equation*}
		\begin{split}
			&\int_{\Phi_{1}^{-1}(1)}^{\infty}\frac{\varphi_{0}^{-1}(t)[\log(e+\Phi_{1}(t))]^{4r'-4}}{t^{2}
				\log(e+t)}dt\\
			&\lesssim\int_{\Phi_{1}^{-1}(1)}^{\infty}\frac{1}{t\log(e+t)[\log(e+\log(e+t))]^{1+\varepsilon}}dt
			\\
			&=\int_{\log(e+\Phi_{1}^{-1}(1))}^{\infty}\frac{e^\mu}{(e^\mu-e)\mu[\log(e+\mu)]
				^{1+\varepsilon}}d\mu\\
			&\lesssim\int_{\log(e+\Phi_{1}^{-1}(1))}^{\infty}\frac{d\mu}{\mu(\log\mu)^{1+\varepsilon}}
			\sim\frac{1}{\varepsilon}.
		\end{split}
	\end{equation*}
	
	For $j=1$, noticing that
	\begin{equation*}
		\begin{split}
			&\int_{1}^{\infty}\frac{\varphi_{1}^{-1}(t)[\log(e+(\log(e+t))^2)]^{r'}}{t^2[\log(e+t)]^{3-2r'}}
			dt\\
			&\quad\lesssim\int_{1}^{\infty}\frac{\varphi_{1}^{-1}(t)}{t^2[\log(e+t)]^{3-3r'}}dt,
		\end{split}
	\end{equation*}
	and taking $\varphi_1(t)=t(\log(e+t))^{3r'-2}[\log(e+\log(e+t))]^{1+\varepsilon}$, we have
	\begin{equation*}
		\begin{split}
			&\int_{1}^{\infty}\frac{\varphi_{1}^{-1}(t)[\log(e+(\log(e+t))^2)]^{r'}}{t^2[\log(e+t)]^{3-2r'}}dt\\
			&\quad\lesssim\int_{1}^{\infty}\frac{1}{t\log(e+t)[\log(e+\log(e+t))]^{1+\varepsilon}}dt
			\lesssim\frac{1}{\varepsilon}.
		\end{split}
	\end{equation*}
	Using that $\log(e+\varphi_0(t))\lesssim \log(e+t)$, we deduce that
	$$\Phi_1\circ\varphi_0(t)\lesssim\varphi_0(t)\log(e+t)=t(\log(e+t))^{4r'-3}
	[\log(e+\log(e+t))]^{1+\varepsilon}.$$
	Combining with $\Phi_0\circ\varphi_1(t)=t(\log(e+t))^{3r'-2}[\log(e+\log(e+t))]^{1+\varepsilon}$ and $$\sum_{k=1}^{\beta_n}\frac{4^{k(r'-1)}\varphi_{0}^{-1}\circ\Phi_{1}^{-1}(1/\alpha_k)}
	{\Phi_{1}^{-1}(1/\alpha_k)}\lesssim \frac{1}{\varepsilon},$$
	we conclude that
	\begin{align*}
		\omega(\{x\in\mathbb{R}^n:\mathcal{V}_{\rho}(\mathcal{T}_bf)(x)>\lambda\})
		&\leq C(n,r',q_0,\|\mathcal{V}_\rho(\mathcal{T})\|_{L^{q_0}\rightarrow L^{q_0}})\frac{1}{\varepsilon}\int_{\mathbb{R}^n}\psi
		\Big(\frac{\|b\|_{\rm BMO}|f(x)|}{\lambda}\Big)\\
		&\quad\times M_{L(\log L)^{4r'-3}(\log\log L)^{1+\varepsilon}}\omega(x)dx.
	\end{align*}
	
	Similarly, take $\varphi_0(t)=t(\log(e+t))^{4r'-4+\epsilon}$ and $\varphi_1(t)=t(\log(e+t))^{3r'-2+\epsilon}$, we get
	\begin{align}\label{new conclusion}
		&\omega(\{x\in\mathbb{R}^n:\mathcal{V}_{\rho}(\mathcal{T}_bf)(x)>\lambda\})\nonumber\\
		&\quad\leq C(n,r',q_0,\|\mathcal{V}_\rho(\mathcal{T})\|_{L^{q_0}\rightarrow L^{q_0}})\frac{1}{\varepsilon}\int_{\mathbb{R}^n}\psi
		\Big(\frac{\|b\|_{\rm BMO}|f(x)|}{\lambda}\Big)\\
		&\qquad\times M_{L(\log L)^{4r'-3+\epsilon}}\omega(x)dx.\nonumber
	\end{align}
	For any $t\geq1,\alpha>0$, since $\log t\leq\frac{t^\alpha}{\alpha}$, we have
	$$\frac{1}{\varepsilon}M_{L(\log L)^{4r'-3+\epsilon}}(x)\leq c\frac{1}{\varepsilon}\frac{1}{\alpha^{(4r'-3)+\varepsilon}}
	M_{L^{1+(4r'-3+\varepsilon)\alpha}}\omega(x).$$
	Hence, let $\alpha(4r'-3+\varepsilon)=\frac{1}{\tau_n[\omega]_{A_\infty}}$, by Lemma $\ref{lm 2.1}$,
	\begin{align*}
		&\frac{1}{\varepsilon}M_{L(\log L)^{4r'-3+\epsilon}}\omega(x)\\
		&\quad\leq c\frac{1}{\varepsilon}\frac{1}{\alpha^{(4r'-3)+\varepsilon}}
		M_{L^{1+\alpha(4r'-3+\varepsilon)}}\omega(x)\\
		&\quad= c\frac{1}{\varepsilon}\Big(\tau_n[\omega]_{A_\infty}
		(4r'-3+\varepsilon)\Big)^{4r'-3+\varepsilon}M_{{1+\frac{1}{\tau_n[\omega]_
					{A_\infty}}}}\omega(x)\\
		&\quad\leq c_{n,r}\frac{1}{\varepsilon}[\omega]_{A_\infty}^{4r'-3+\varepsilon}
		M_{{1+\frac{1}{\tau_n[\omega]_{A_\infty}}}}\omega(x).
	\end{align*}
	Finally, choose~$\varepsilon=\frac{1}{\log(e+[\omega]_{A_\infty})}$, then $[\omega]_{A_\infty}^\epsilon\leq e$ and it implies that
	\begin{align*}
		\frac{1}{\varepsilon}M_{L(\log L)^{4r'-3+\epsilon}}\omega(x)
		\leq c_{n,r}[\omega]_{A_\infty}^{4r'-3}[\log(e+[\omega]_{A_\infty})]M\omega(x).
	\end{align*}
	This together with $(\ref{new conclusion})$, we complete the proof of Theorem $\ref{Theorem1.5}$.
\end{proof}

\section{Including result}

In this last section, we will apply the sparse domination obtained in Theorem $\ref{Theorem1.1}$ to present the local exponential decay estimates of variation operators. Let us recall some backgrounds before stating our contributions. It was known that Coifman and Fefferman applied good-$\lambda$ technique to obtain
$$\|T^\ast f\|_{L^p(\omega)}\leq c\|Mf\|_{L^p(\omega)},$$
where $T^\ast$ is the maximal Calder\'{o}n-Zygmund operator and $\omega\in A_{\infty}$. The method was relied heavily on the estimate
$$|\{x\in \mathbb{R}^n: T^\ast f(x)>2\lambda,Mf(x)\leq\lambda\gamma\}|\leq c\gamma|\{x\in \mathbb{R}^n: T^\ast f(x)>\lambda\}|.$$
To show the above estimate, it suffices to study the following local etimate
$$|\{x\in Q: T^\ast f(x)>2\lambda,Mf(x)\leq\lambda\gamma\}|\leq c\gamma|Q|,$$
where $Q$ is a Whitney cube and $\text {supp}~ f\subset Q$. In 1993, Buckley \cite{Bu} obtained an exponential decay in $\gamma$ in studying the quantitative weighted estimates for Calder\'{o}n-Zygmund operators,
\begin{equation}\label{(1.7)}
	|\{x\in Q: T^\ast f(x)>2\lambda,Mf(x)\leq\lambda\gamma\}|\leq ce^{-\frac{c}{\gamma}}|Q|.
\end{equation}
Based on the result above, it yields that
$$\|Tf\|_{L^p(\omega)}\leq cp[\omega]_{A_\infty}\|Mf\|_{L^p(\omega)},$$
which is also a key to the $L\log L$ estimate obtained in \cite{LOP}. Later on, Karagulyan \cite{Ka} improved $(\ref{(1.7)})$ by giving the below estimate:
$$|\{x\in Q:T^\ast f(x)>tMf(x)\}|\leq ce^{-\alpha t}|Q|.$$
The above estimate was then extended to other operators by Ortiz-Caraballo, P\'{e}rez and Rela in \cite{OCPR}. We also refer readers to \cite{CJa} for its application to the quantitative $C_p$ estimate for Calder\'{o}n-Zygmund operators.

Now, by Theorem $\ref{Theorem1.1}$ and employing the arguments in \cite{CLPR,IR}, we may obtain:

\begin{theorem}\label{Theorem1.6}\,\
	Let $1<r\leq\infty$, $\rho>2$, $b\in {\rm BMO}(\mathbb{R}^n)$. Assume that $K(x,y)\in \mathcal{H}_r$ and $(\ref{(1.4)})$. Let $\mathcal{T}$ and $\mathcal{T}_{b}$ be given by $(\ref{(1.2)})$ and $(\ref{(1.3)})$, respectively. If $\mathcal{V}_\rho(\mathcal{T})$ is bounded on $L^{q_0}(\mathbb{R}^n,dx)$ for some $1<q_0<\infty$, then for $\text {supp}~f\subset Q$, there exist constants $c_1$, $c_2$, $c_3$ and $c_4$ such that
	\begin{align*}
		&|\{x\in Q:\mathcal{V}_\rho(\mathcal{T}f)(x)>tM_{r'}(f)(x)\}|\leq c_1e^{-c_2t}|Q|,\\
		&|\{x\in Q:\mathcal{V}_\rho(\mathcal{T}_bf)(x)>tM_{L^{r'}(\log L)^{r'}}(f)(x)\}|\leq c_3e^{-\sqrt{c_4t/\|b\|_{\rm BMO}}}|Q|.
	\end{align*}
\end{theorem}
This provides an extension of the corresponding results for singular integrals and commutators in \cite{IR,OCPR} to the variation operators. As far as we are concerned, these results are completely new, since no local exponential decay estimates of variation operators have been considered before.

\subsection*{Acknowledgements} The authors would like to thank Professor Kangwei Li for
a helpful discussion on the subject of this article.

\end{document}